\documentclass[10pt]{amsart}
\usepackage{amsmath,amssymb, amsthm, amsbsy}
\usepackage[dvips]{graphicx}
\usepackage[usenames,dvipsnames]{pstricks}
\usepackage{epsfig}
\usepackage{pst-grad}
\usepackage{pst-plot}

\newcommand{\R}{{\mathbb R}}
\newcommand{\N}{{\mathbb N}}
\newcommand{\Z}{{\mathbb Z}}

\newcommand{\sgn}{\mathop{\rm sgn}}
\newcommand{\A}{{\mathbf A}}

\renewcommand{\geq }{\geqslant}

\renewcommand{\leq }{\leqslant}
\newenvironment{pf}{\noindent{\sc Proof}.\enspace}{\hfill\qed\medskip}

\numberwithin{equation}{section}
\newtheorem{Theorem}{Theorem}[section]
\newtheorem{Corollary}[Theorem]{Corollary}
\newtheorem{Lemma}[Theorem]{Lemma}
\newtheorem{Proposition}[Theorem]{Proposition}
\theoremstyle{definition}

\newtheorem{remark}[Theorem]{Remark}

\begin{document}
\title[$2D$-Schr\"odinger with critical potentials]{Time decay of scaling
invariant electromagnetic Schr\"odinger equations on the plane}
\author{Luca Fanelli}
\address{Luca Fanelli: SAPIENZA Universit$\grave{\text{a}}$ di Roma, P.le A.
Moro 5, 00185, Roma, Italy}
\email{fanelli@mat.uniroma1.it}
\author{Veronica Felli}
\address{Veronica Felli: Universit$\grave{\text{a}}$ di Milano Bicocca,
Dipartimento di Matematica e applicazioni, Via Cozzi 55, 20125, Milano, Italy%
}
\email{veronica.felli@unimib.it}
\author{Marco A. Fontelos}
\address{Marco Antonio Fontelos: ICMAT-CSIC, Ciudad Universitaria de
Cantoblanco. 28049, Madrid, Spain}
\email{marco.fontelos@icmat.es}
\author{Ana Primo}
\address{Ana Primo: Departamento de Matem\'aticas, Universidad Aut\'onoma de Madrid,
28049, Madrid, Spain.}
\email{ana.primo@uam.es}

\begin{abstract}
We prove the sharp $L^1-L^\infty$ time-decay estimate for the $2D$
-Schr\"odinger equation with a general family of scaling critical
electromagnetic potentials.
\end{abstract}

\thanks{
The first author was supported by the Italian project FIRB 2012:
``Dispersive dynamics: Fourier Analysis and Variational Methods''.
The second author  was partially supported by the P.R.I.N. 2012 grant
  ``Variational and perturbative aspects of nonlinear differential
  problems''.
The third author was supported by the Spanish project ``MTM2011-26016''. The fourth author was supported by the Spanish project ``MTM2010-18128''.}

\date{May 7, 2014}
\maketitle



\section{Introduction}

\label{sec:intro} Let us consider an electromagnetic Schr\"odinger equation
of the type
\begin{equation}  \label{eq:schro}
iu_t=\left(-i\nabla+ \dfrac{{\mathbf{A}}\big(\frac{x}{|x|}\big)} {|x|}
\right)^{\!\!2} u+ \dfrac{a\big(\frac{x}{|x|}\big)}{|x|^2}\,u,
\end{equation}
where $N\geq2$, $u=u(x,t):{\mathbb{R}}^{N+1}\to{\mathbb{C}}$, $a\in
W^{1,\infty}({\mathbb{S}}^{N-1}, {\mathbb{R}})$, ${\mathbb{S}}^{N-1}$ denotes
the unit circle, and ${\mathbf{A}}\in W^{1,\infty} ({\mathbb{S}} ^{N-1},{\mathbb{R}}
^N) $ is a transversal vector field, namely
\begin{equation}  \label{eq:transversality}
{\mathbf{A}}(\theta)\cdot\theta=0 \quad \text{for all }\theta\in {\mathbb{S}}
^{N-1}.
\end{equation}
We always denote by $r:=|x|$, $\theta=x/|x|$, so that $x=r\theta$. Notice
that the potentials ${{\mathbf{A}}}/|x|$ and $a/|x|^2$ preserve the natural
scaling $u_\lambda(x,t):=u(x/\lambda, t/\lambda^2)$ of the free
Schr\"odinger equation, and consequently they show a critical behavior with
respect to several phenomena.

In \cite{FFFP}, we started a program based on the connection between the
Schr\"odinger flow $e^{it\mathcal{L}_{{\mathbf{A}},a}}$, generated by the
hamiltonian
\begin{equation} \label{eq:hamiltonian}
  \mathcal{L}_{{\mathbf{A}},a}:=\left(-i\nabla+
    \dfrac{{\mathbf{A}}\big(\frac{x}{|x|}\big)} {|x|} \right)^{\!\!2}+
  \dfrac{a\big(\frac{x}{|x|}\big)}{|x|^2},
\end{equation}
and the spectral properties of the spherical operator $L_{{\mathbf{A}},a}$,
defined by
\begin{equation}  \label{eq:laplacebeltrami}
L_{{\mathbf{A}},a} =\big(-i\,\nabla_{\mathbb{S}^{N-1}}+{\mathbf{A}}\big)^2+a(\theta),
\end{equation}
where $\nabla_{\mathbb{S}^{N-1}}$ is the spherical gradient on the unit
sphere $\mathbb{S}^{N-1}$. In order to describe the project, let us start by
reviewing some well known facts in classical spectral theory.

The spectrum of the operator $L_{{\mathbf{A}},a}$ is formed by a
diverging sequence of real eigenvalues with finite multiplicity
$\mu_1({\mathbf{A}}
,a)\leq\mu_2({\mathbf{A}},a)\leq\cdots\leq\mu_k({\mathbf{A}},a)\leq\cdots$
(see e.g. \cite[Lemma A.5]{FFT}), where each eigenvalue is repeated
according to its multiplicity. Moreover we have that
$\lim_{k\to\infty}\mu_k({\mathbf{A}},a)=+\infty$.  To each $k\geq1$,
we can associate a
$L^{2}\big({\mathbb{S}}^{N-1},{\mathbb{C}}\big)$-normalized
eigenfunction $\psi_k$ of the operator $L_{{\mathbf{A}},a}$ on
$\mathbb{S} ^{N-1}$ corresponding to the $k$-th eigenvalue
$\mu_{k}({\mathbf{A}},a)$, i.e.  satisfying
\begin{equation}  \label{eq:angular}
\begin{cases}
L_{{\mathbf{A}},a}\psi_{k}=\mu_k({\mathbf{A}},a)\,\psi_k(\theta), & \text{in
}{\mathbb{S}}^{N-1}, \\[3pt]
\int_{{\mathbb{S}}^{N-1}}|\psi_k(\theta)|^2\,dS(\theta)=1. &
\end{cases}
\end{equation}
In particular, if $N=2$, the functions $\psi_k$ are one-variable $2\pi$
periodic functions, i.e. $\psi_k(0) = \psi_k(2\pi)$. Since the
eigenvalues $\mu_k({\mathbf{A}},a)$ are repeated according to their
multiplicity, exactly one eigenfunction $\psi_k$ corresponds to
each index $k\geq1$. We can choose the functions $\psi_k$ in such
a way that they form an orthonormal basis of $L^2({\mathbb{S} }^{N-1},{\mathbb{C}})$. We also introduce the numbers
\begin{equation}  \label{eq:alfabeta}
\alpha_k:=\frac{N-2}{2}-\sqrt{\bigg(\frac{N-2}{2}\bigg)^{\!\!2}+\mu_k({\
\mathbf{A}},a)}, \quad \beta_k:=\sqrt{\left(\frac{N-2}{2}\right)^{\!\!2}+
\mu_k({\mathbf{A}},a)},
\end{equation}
so that $\beta_{k}=\frac{N-2}{2}-\alpha_{k}$, for $k=1,2,\dots$, which will
come into play in the sequel.

Under the condition
\begin{equation}  \label{eq:hardycondition}
\mu_1({\mathbf{A}},a)>-\left(\frac{N-2}{2}\right)^{\!\!2}
\end{equation}
the quadratic form associated to $\mathcal{L}_{{\mathbf{A}},a}$ is positive
definite (see \cite[Section2]{FFFP} and \cite{FFT}); this implies that the
hamiltonian $\mathcal{L}_{{\mathbf{A}},a}$ is a symmetric semi-bounded
operator on $L^2({\mathbb{R}}^N;{\mathbb{C}})$ which then admits a
self-adjoint extension (the \emph{Friedrichs extension} which will be still
denoted as $\mathcal{L}_{{\mathbf{A}},a}$) with domain
\begin{equation}  \label{eq:domain}
\mathcal{D}(\mathcal{L}_{\mathbf{A},a}):=\left\{ f\in H^1_*({\mathbb{R}}
^N):\ \mathcal{L}_{\mathbf{A},a}u\in L^2({\mathbb{R}}^N\right\},
\end{equation}
where $H^1_*({\mathbb{R}}^N)$ is the completion of $C^{\infty}_{\mathrm{c} }(
{\mathbb{R}}^N\setminus\{0\},{\mathbb{C}})$ with respect to the norm
\begin{equation*}
\|\phi\|_{H^1_*({\mathbb{R}}^N)}=\bigg(\int_{{\mathbb{R}}^N}\bigg(
|\nabla\phi(x)|^2+ \frac{|\phi(x)|^2}{|x|^2}+|\phi(x)|^2\Big)\bigg) \,dx
\bigg)^{\!\!1/2}.
\end{equation*}
From the classical Hardy inequality (see e.g. \cite{GP,HLP}),
$H^1_*({\mathbb{R}}^N)=H^1({\mathbb{R}}^N)$ with equivalent norms if
$N\geq 3$, while $H^1_*({\mathbb{R}}^N)$ is strictly smaller than
$H^1({\mathbb{R}}^N)$ if $N=2$. Furthermore, from condition
\eqref{eq:hardycondition} and \cite[Lemma 2.2]{FFT}, it follows that $H^1_*({\mathbb{R}}^N)$ coincides
with the space obtained by completion of $C^{\infty}_{\mathrm{c}
}({\mathbb{R}}^N\setminus\{0\},{\mathbb{C}})$ with respect to the norm naturally
associated to the operator $\mathcal{L}_{\mathbf{A},a}$, i.e.
\begin{equation*}
  \bigg(\int_{{\mathbb{R}}^N} \bigg[ \bigg|
  \bigg(\nabla+i\,\frac {{\mathbf{A}}
    \big({x}/{|x|}\big)}{|x|}\bigg)u(x)\bigg|^2
  +\frac{a\big({x}/{|x|}\big)}{|x|^2}|u(x)|^2+|u(x)|^2\bigg]\,dx\bigg)^{\!\!1/2}.
\end{equation*}

We notice that $\mathcal{L}_{{\mathbf{A}},a}$ could be not essentially
self-adjoint. For example, in the case ${\mathbf{A}}\equiv0$, from a
theorem due to Kalf, Schmincke, Walter, and W\"ust \cite{kwss} and
Simon \cite{simon73} (see also \cite[Theorems X.11 and
X.30]{reedsimon}, \cite{FMT}, and \cite{FMT2} for non constant $a$),
it is known that $\mathcal{L}_{{\mathbf{0}},a}$ is essentially
self-adjoint if and only if $\mu_1(\mathbf{0},a)\geq -\big(\frac{N-2}{2}\big)^{2}+1$ and, consequently, admits a
unique self-adjoint extension, which is given by the Friedrichs
extension; otherwise, i.e. if $\mu_1(\mathbf{0},a)<
-\big(\frac{N-2}{2}\big)^{2}+1$, $\mathcal{L}_{{\mathbf{0}},a}$ is not essentially self-adjoint and
admits many self-adjoint extensions, among which the Friedrichs
extension is the only one whose domain is included in the domain of
the associated quadratic form (see also \cite[Remark
2.5]{duyckaerts}).

The Friedrichs extension $\mathcal{L}_{{\mathbf{A}},a}$ naturally
extends to a self adjoint operator on the dual of
$\mathcal{D}(\mathcal{L}_{\mathbf{A},a})$ and the unitary group of
isometries $e^{-it\mathcal{L}_{{\mathbf{A }},a}}$ generated by
$-i\mathcal{L}_{{\mathbf{A}},a}$ extends to a group of isometries on
the dual of $\mathcal{D}(\mathcal{L}_{\mathbf{A},a})$ which will be
still denoted as $e^{-it\mathcal{L}_{{\mathbf{A }},a}}$ (see
\cite{Cazenave}, Section 1.6 for further details). Then for every
$u_0\in L^2({  \mathbb{R}}^N)$, $u(\cdot,t)=e^{-it\mathcal{L}_{{\mathbf{A
      }},a}}u_0(\cdot)$ is the unique solution to the problem
\begin{equation*}
\begin{cases}
  u\in \mathcal{C}(\mathbb{R},L^2({\mathbb{R}}^N))\cap
  C^1({\mathbb{R}},
(\mathcal{D}(\mathcal{L}_{\mathbf{A},a}))^{\star}), \\
  iu_t=\mathcal{L}_{{\mathbf{A }},a}u, \\
  u(0)=u_0.
\end{cases}
\end{equation*}
Now, by means of \eqref{eq:angular} and \eqref{eq:alfabeta} define the
following kernel:
\begin{equation}  \label{nucleo}
K(x,y)=\sum\limits_{k=-\infty}^{\infty }i^{-\beta _{k}}j_{-\alpha
_{k}}(|x||y|)\psi _{k}\big(\tfrac{x}{|x|}\big)\overline{\psi _{k}\big(\tfrac{
y}{|y|}\big)},
\end{equation}
where
\begin{equation*}
j_{\nu }(r):=r^{-\frac{N-2}{2}}J_{\nu +\frac{N-2}{2}}(r)
\end{equation*}
and $J_{\nu }$ denotes the usual Bessel function of the first kind
\begin{equation*}
J_{\nu }(t)=\bigg(\frac{t}{2}\bigg)^{\!\!\nu }\sum\limits_{k=0}^{\infty }
\dfrac{(-1)^{k}}{\Gamma (k+1)\Gamma (k+\nu +1)}\bigg(\frac{t}{2}\bigg)
^{\!\!2k}.
\end{equation*}
In the main result of \cite{FFFP} we prove that, if $a\in
L^{\infty }({\mathbb{S}}^{N-1},{\mathbb{R}})$ and ${\ \mathbf{A}}\in C^{1}({
\mathbb{S}}^{N-1},{\mathbb{R}}^{N})$ are such that \eqref{eq:transversality}
and \eqref{eq:hardycondition} hold, then
\begin{equation}  \label{eq:representation}
e^{-it\mathcal{L}_{\mathbf{A},a}}u_0(x)=\frac{e^{\frac{i|x|^{2}}{4t}}}{
i(2t)^{{N}/{2}}}\int_{{\mathbb{R}} ^{N}}K\bigg(\frac{x}{\sqrt{2t}},\frac{y}{
\sqrt{2t}}\bigg)e^{i\frac{|y|^{2}}{ 4t}}u_{0}(y)\,dy,
\end{equation}
for any $u_0\in L^2({\mathbb{R}}^N)$.

Apart from the interest in itself, formula \eqref{eq:representation}
provides a quite solid tool to obtain quantitative informations for the flow
$e^{-it\mathcal{L}_{\mathbf{A},a}}u_0(x)$ by the analytical study of the
kernel $K(x,y)$. In particular, if
\begin{equation}  \label{eq:infty}
\sup_{x,y\in\mathbb{R}^N}\left|K(x,y)\right|<\infty
\end{equation}
holds, one automatically obtains by \eqref{eq:representation} the time-decay
estimate
\begin{equation}  \label{eq:decay}
\left\|e^{-it\mathcal{L}_{\mathbf{A},a}}u_0(\cdot)\right\|_{L^\infty}
\lesssim |t|^{-\frac N2}\|u_0(\cdot)\|_{L^1}.
\end{equation}
In \cite{FFFP}, we are able to prove \eqref{eq:infty} (and consequently
\eqref{eq:decay}) in two concrete situations:

\begin{itemize}
\item the \textit{Aharonov-Bohm} potential: $a\equiv0$, $\mathbf{A}
(x)=\alpha \left(-\frac{x_2}{|x|}, \frac{x_1}{|x|}\right)$, for $\alpha\in
\mathbb{R}$, in dimension $N=2$;

\item the positive \textit{inverse square} potential:
  $\mathbf{A}\equiv0$, $a\in\mathbb{R}$, $a>0$.
\end{itemize}

In both cases, the spectrum of $L_{\mathbf{A},a}$ is explicit, together with
a complete set of orthonormal eigenfunctions (spherical harmonics or phase
transformations of themselves). These examples give a positive contribution
to the recent literature about the topic, which never included before
potentials with the critical homogeneity as the ones in \eqref{eq:schro}
(see e.g. \cite{BG, BPSTZ1, BPSTZ, EGS1, EGS2, DF1, DF2, DF3, DFVV, GST, G,
GS, MMT, PSTZ, RZ, RS, S, ST, Ste, W1, W2, Y1, Y2, Y3, Y4}). Moreover, it is
well known that these potentials represent a threshold between the validity
and the failure of global (in time) dispersive estimates, as proved in \cite{FG, GVV}. Recently, Grillo and Kovarik \cite{GK} gave a proof of sharp time-decay estimates in the case of the Aharonov-Bohm potential, combined with a compactly supported electric potential, in dimension 2, proving also an interesting remark regarding the connection of diamagnetism with improvement of decay, in suitable weighted spaces.

The aim of this paper is to prove that estimate \eqref{eq:decay} holds, in
space dimension $N=2$, for a general family of potentials of the same kind
as in \eqref{eq:schro}. Our main result is the following.

\begin{Theorem}
  \label{thm:main} Let $N=2$, $a\in
  W^{1,\infty}(\mathbb{S}^1,\mathbb{R})$, $\mathbf{A}\in
  W^{1,\infty}(\mathbb{S}^1,\mathbb{R}^2)$ satisfying
  \eqref{eq:transversality} and $\mu_1(\A,a)>0$, and
  $\mathcal{L}_{\mathbf{A},a}$ be given by
  \eqref{eq:hamiltonian}. Then, for any $u_0\in
  L^2({\mathbb{R}}^N)\cap L^1({ \mathbb{R}}^N)$, the following
  estimate holds:
\begin{equation}  \label{eq:main}
\left\|e^{-it\mathcal{L}_{\mathbf{A},a}}u_0(\cdot)\right\|_{L^\infty} \leq
\frac{C}{|t|}\|u_0(\cdot)\|_{L^1},
\end{equation}
for some $C=C({\mathbf{A}},a)>0$ which does not depend on $t$ and $u_0$.
\end{Theorem}

As remarked above, the proof of Theorem \ref{thm:main} consists in
showing that the kernel $K(x,y)$ in \eqref{nucleo} is uniformly
bounded. The main difficulty is to obtain this information for the
queues of the series in \eqref{nucleo}. In order to do this, we need
to obtain the precise asymptotic behavior in $k$ of the set of
eigenvalues and eigenfunctions of the problem \eqref{eq:angular}: this
is the topic of Section \ref{sec:spectral} below. Once this is done,
the proof of Theorem \ref{thm:main} will be obtained, in Section
\ref{sec:proof}, by suitably comparing the kernel $K$ with the
analogous in the case of an Aharonov-Bohm potential with the same
average as the potential $\mathbf{A}$ on the sphere $\mathbb{S}^1$.

\textbf{Acknowledgements.} The authors wish to thank Carlos Villegas-Blas
for addressing them to several useful references about the topic of Section
\ref{sec:spectral}.

\section{Spectral properties of spherical laplacians}

\label{sec:spectral} The fundamental tool which we need in order to prove
Theorem \ref{thm:main} is the knowledge of the spectral properties of the
operator $L_{\mathbf{A},a}$ defined by \eqref{eq:laplacebeltrami}. Roughly
speaking, we need to obtain informations concerning the asymptotic behavior of
the eigenvalues $\mu_k$ and the eigenfunctions $\psi_k$ in the eigenvalue
problem \eqref{eq:angular}, as $k\to\infty$.

An extensive literature has been devoted, in the recent years, to this kind
of problems (see e.g. \cite{Gu, TV, TW, Weinstein} and the references
therein). Since we did not find sufficiently explicit results regarding
general electromagnetic Laplace operators on the 1D-sphere $\mathbb{S}^1$,
we need to show here Lemma \ref{lem:spectral} below, which is possibly of
independent interest.

Before starting to settle the eigenvalue problem, we find convenient
to briefly sketch the well known consequences which the introduction
of lower order terms produces on the spectrum of the spherical
Laplacian.

Let us denote by $L_0:=-\Delta_{\mathbb S^1}$. Being the inverse of a
compact operator on $L^2(\mathbb S^1)$, with form domain $H^1(\mathbb
S^1)$, $L_0$ has purely discrete spectrum which accumulates at
positive infinity. The explicit form is
\begin{equation*}
  \sigma(L_0)=\{k^2\}_{k=0,1,\dots}.
\end{equation*}
The $k^{\text{th}}$-eigenvalue has multiplicity $2$,
and the eigenfunctions are combinations of sines and cosines.

The introduction of a 0-order term produces a spectral shift,
depending on the average of the potentials, and the formation of
clusters of eigenvalues around the free ones (Stark's effect), if the
potential is not constant.  More precisely, the eigenvalues of
the operator $L_a:=-\Delta_{\mathbb S^1}+a(\theta)$ are of the form
\begin{equation*}
\lambda_k=k^2+\frac{1}{2\pi}\int_0^{2\pi}a(s)\,ds + (\text{rest}),
\end{equation*}
where the rest, depending on $k$ and on the potential $a$, decays with
order $1/k$ as $k$ tends to infinity. For the eigenfunctions $\psi_k$
a similar behavior occurs; for large $k$, $\psi_k$ looks more and more
like a spherical harmonic plus a rest which decays as $k$ tends to
$+\infty$ (see e.g. \cite{B, Gu, kwss, TV, TW} and appendix \ref{sec:asympt-eigenv-eigenf}).

On the other hand, for a purely magnetic potential, a splitting occurs
on each eigenvalue. The most famous (and descriptive) example is given
by the Aharonov-Bohm potential, namely $a\equiv0$,
$\mathbf{A}(x,y)=\mathbf{A}_{ab}(x,y)=\alpha\left(-\frac{x_2}{|x|^2},
  \frac{x_1}{|x|^2}\right)$, with $\alpha\in\mathbb{R}$: in this case,
the complete
set of eigenvalues and eigenfunctions of problem \eqref{eq:angular}
can be computed explicitly, and reads as
\begin{align}  \label{eq:eigfab}
\lambda_{k}^{ab} =(k+\alpha)^2,&\quad k\in \Z \\
\phi_k ^{ab} (\theta)  =\frac1{\sqrt{2\pi}}e^{ik\theta},&\quad k\in \Z.
\label{eq:eigfab2}
\end{align}

It is hence quite natural to expect that, in the general case of the operator $L_{\mathbf{A},a}$, the picture is a superposition of the two previously mentioned ones. We did not find in the literature a result written in the generality of Lemma \ref{lem:spectral} below, so that we found convenient to state and prove it in this manuscript.

We recall that, by classical spectral theory, the spectrum of $L_{{\mathbf{A}},a}$ is formed by a countable family
of real eigenvalues with finite multiplicity
$\{\mu_k:k\geq1\}$ enumerated in such a way that
\[
\mu_1\leq \mu_{2}\leq\dots,
\]
where  each eigenvalue is repeated according to its  multiplicity. Moreover we have that
$\lim_{k\to\infty}\mu_k=+\infty$.

Let $A:[0,2\pi]\to\R$ be defined as $A(\theta)=\A(\cos \theta,\sin  \theta)\cdot(-\sin
 \theta,\cos  \theta)$, so that, by assumption \eqref{eq:transversality}
\begin{equation}\label{eq:AA}
\A(\cos  \theta,\sin \theta)=A(\theta)(-\sin \theta,\cos \theta), \quad \theta\in[0,2\pi].
\end{equation}
Furthermore, identifying functions defined on ${\mathbb S}^1$ with
$2\pi$-periodic functions, the operator $L_{{\mathbf{A}},a}$ can be
identified with the following operator $\frak L_{A,a}$ acting on
$2\pi$-periodic functions
\[
\frak L_{A,a}\varphi(\theta)=
-\varphi''(\theta)+[a(\theta)+A^2(\theta)-iA'(\theta)] \varphi(\theta)-2iA(\theta)
\varphi'(\theta).
\]
The main result of this section is the following asymptotic expansion
of eigenvalues and eigenfunctions of the operator $L_{{\mathbf{A}},a}$
under the non-resonant assumption that the magnetic potential does not have
half-integer or integer circulation. The case of half-integer or
integer circulation can be reduced through suitable transformations to
the magnetic-free problem, for which analogous expansions hold, see
Remark \ref{rem:gurarie} and appendix \ref{sec:asympt-eigenv-eigenf}.

\begin{Lemma}\label{lem:spectral}
  Let  $a\in W^{1,\infty}({\mathbb S}^1)$, $\widetilde
a=\frac1{2\pi}\int_0^{2\pi}a(s)\,ds$, $A\in W^{1,\infty}({\mathbb
  S}^1)$ such that
\begin{equation}\label{eq:usek}
\widetilde
A=\frac1{2\pi}\int_0^{2\pi}A(s)\,ds\not\in
\frac12\Z.
\end{equation}
Then there exist  $k^*,\ell\in\N$ such that $\{\mu_k:\ k>
k^*\}=\{\lambda_j:\ j\in\Z,\ |j|\geq \ell\}$,
\[\sqrt{\lambda_j-\widetilde a}=(\sgn j)
\Big(
\widetilde A-\left\lfloor \widetilde A+\tfrac12\right\rfloor
\Big)
+|j|+O\big(\tfrac1{|j|^3}\big),\quad\text{as }|j|\to+\infty
\]
and
\begin{equation}\label{eq:eigenvalues}
\lambda_j=\widetilde a+
\Big(j+\widetilde A-\big\lfloor \widetilde A+\tfrac12\big\rfloor
\Big)^2+O\big(\tfrac1{j^2}\big),\quad\text{as }|j|\to+\infty.
\end{equation}
Furthermore, for all $j\in\Z$, $|j|\geq\ell$, there exists
 a $L^{2}\big({\mathbb{S}}^{1},{\mathbb{C}}\big)$-normalized
eigenfunction  $\phi_j$ of the operator $L_{{\mathbf{A}},a}$ on $\mathbb{S}^{1}$
corresponding to
 the  eigenvalue $\lambda_j$ such that
\begin{align}\label{eq:eigenfunctions}
  \phi_j(\theta)=
\frac1{\sqrt{2\pi}}e^{-i\big([\widetilde A+1/2]\theta+\int_0^\theta
  A(t)\,dt\big)}\Big(e^{i(\widetilde A+j)\theta}+R_j(\theta)\Big),
\end{align}
where $\|R_j\|_{L^\infty({\mathbb S}^1)}= O\big(\tfrac1{|j|^3}\big)$ as
$|j|\to\infty$. In the above formula $\lfloor\cdot\rfloor$ denotes the floor function
$\lfloor x\rfloor=\max\{k\in\Z:k\leq x\}$.
\end{Lemma}

Lemma \ref{lem:spectral} can be interpreted as follows: asymptotically
in $k $, eigenvalues and eigenfunctions of \eqref{eq:angular} for
$L_{\mathbf{A},a}$ are comparable with the ones in the Aharonov-Bohm
case (see \eqref{eq:eigfab}, \eqref{eq:eigfab2} above), by means of
\eqref{eq:eigenvalues}, \eqref{eq:eigenfunctions}.

The proof of Lemma \ref{lem:spectral} is based on the idea of reducing the eigenvalue problem \eqref{eq:angular} to another magnetic-free problem,
with different boundary conditions, by gauge transformation; this is in fact
possible, since $\mathbf{A}(\cos\theta,\sin\theta)$ just depends on the 1D-variable $\theta$. More precisely,
we observe that the gauge transformation
\begin{equation*}
\psi(\theta) \to e^{-i\int_0^\theta A(s)\,ds}\psi(\theta)
\end{equation*}
transforms the eigenvalue problem \eqref{eq:angular} into the new problem
\begin{equation}  \label{eq:eigprob}
\begin{cases}
-\frac{d^{2}\psi}{d\theta ^{2}}+a(\theta )\psi=\mu _{k}\psi
\\
\psi(0) =e^{-2i\pi \widetilde{ A} }\psi (2\pi ) \\
\psi^{\prime }(0) =e^{-2i\pi \widetilde{ A} }\psi^{\prime }(2\pi ),
\end{cases}
\end{equation}
with non-periodic boundary conditions, where $\widetilde{ A}$ is
defined in \eqref{eq:usek}, which will be analyzed by a usual WKB-strategy.

\begin{remark}
  \label{rem:gurarie} As mentioned above, in the purely electric case
  $\mathbf{A}\equiv0$, \eqref{eq:eigenvalues} is a well known
  information about the cluster distribution of the eigenvalues (see
  e.g. \cite{Gu} and the references therein).  More in general, if
  $\text{dist}(\widetilde{A},\Z)=0$, then the eigenvalue problem
  \eqref{eq:eigprob} reduces to
\begin{equation}
\begin{cases}
-\frac{d^{2}\psi _{k}}{d\theta ^{2}}+a(\theta )\psi _{k}=\mu _{k}\psi _{k}
\\
\psi _{k}(0)=\psi _{k}(2\pi ) \\
\psi _{k}^{\prime }(0)=\psi _{k}^{\prime }(2\pi ),
\end{cases}
\label{eq:borg}
\end{equation}
i.e. the magnetic-free case. For the proof of Lemma \ref{lem:spectral}
in the case $\text{dist}(\widetilde{A},\Z)=0$ we mention a vintage
result by Borg \cite{B} (see also \cite{Gu}) as a standard
reference; for the sake of completeness we sketch a proof of
asymptotics of eigenvalues and eigenfunctions in the purely electric
case in the appendix. Nevertheless, we propose here a proof in the case
$\text{dist}(\widetilde{A},\Z)\neq0,\frac12$, since we did not find in
the literature neither  the analogous to \cite{B} for $\mathbf{A}\neq0$
nor the asymptotic formula for eigenfunctions \eqref{eq:eigenfunctions}, which plays a
fundamental role in the proof of our main theorem (see section
\ref{sec:proof} below).  We propose a proof which is based on a usual
WKB-strategy.
\end{remark}

\subsection{Proof of Lemma \ref{lem:spectral}}

Let us denote
\begin{equation}\label{eq:defAbar}
\bar A=\widetilde A-\left\lfloor \widetilde A+\frac12\right\rfloor,
\end{equation}
so that $\bar A\in [-1/2,1/2)$; we notice that $\widetilde A\in
\frac12\Z$ if and only if $\bar A\in \{-1/2,0\}$. Hence, under
assumption \eqref{eq:usek}, we have that
\begin{equation}\label{eq:usekbar}
\bar A\in \Big(-\frac12,\frac12\Big)\setminus\{0\}.
\end{equation}

\begin{Lemma}\label{l:probequiv}
  Let $a,A\in W^{1,\infty}(0,2\pi)$, $\widetilde
  A=\frac1{2\pi}\int_0^{2\pi}A(s)\,ds$, and $\bar A$ as in
  (\ref{eq:defAbar}), i.e. $\bar A=\widetilde A-\big\lfloor \widetilde
  A+\frac12\big\rfloor$. Then, letting $\A$ as in \eqref{eq:AA}, we have that
\[
\sigma(L_{\A,a})=
\sigma(\frak L_{A,a})=
\sigma(\frak L_{\bar A,a}).
\]
Furthermore, $\varphi$ is an eigenfunction of $\frak L_{A,a}$
associated to the eigenvalue $\mu$ if and only if
$\widetilde\varphi(t)=e^{-i\bar At}e^{i\int_0^tA(s)\,ds}\varphi(t)$ is an
eigenfunction of $\frak L_{\bar A,a}$ associated to $\mu$.
\end{Lemma}
\begin{pf}
The proof follows by direct calculations. We notice that, since $\widetilde A-\bar A\in \Z$,  function
$\widetilde\varphi(t)=e^{-i\bar At}e^{i\int_0^tA(s)\,ds}\varphi(t)$ is $2\pi$-periodic if and only if
$\varphi(t)$ is $2\pi$-periodic.
\end{pf}

\begin{Lemma}\label{l:contraction1}
Let $a\in W^{1,\infty}({\mathbb S}^1)$, $\widetilde
a=\frac1{2\pi}\int_0^{2\pi}a(s)\,ds$,
$\delta>0$,
and
\[
I_\delta=\{\lambda\in\R:\ \mathop{\rm dist}(\sqrt{\lambda-\tilde
  a},\tfrac1 {2}\Z)\geq \delta\}.
\]
There exist $\bar\lambda_\delta>0$ and  $C_\delta>0$ such that for every
$\lambda\in I_\delta$, $\lambda\geq\bar\lambda_\delta$,
there exists $W_\lambda\in C^0({\mathbb S}^1)$ such that
\begin{equation}\label{eq:inteball}
\|W_\lambda\|_{C^0({\mathbb S}^1)}\leq \frac{C_\delta}{\sqrt{\lambda-\tilde
  a}}
\end{equation}
and
\[
T_\lambda(W_\lambda)= W_\lambda,
\]
where $T_\lambda:C^0({\mathbb S}^1)\to C^0({\mathbb S}^1)$ is defined as
\begin{align*}
T_\lambda(W)(\theta)& =
e^{-2\sqrt{\lambda-\widetilde{a}}\theta i}
\frac{\widetilde{a} -a(0)}{2\sqrt{\lambda-\widetilde{a}}
}\\
&\qquad+\frac{ie^{-2\sqrt{\lambda-\widetilde{a}}(\theta+2\pi) i}}{1-e^{-4\sqrt{\lambda-
\widetilde{a}}\pi i}}\int_{0}^{2\pi }e^{2\sqrt{\lambda-\widetilde{a}}
\theta ^{\prime }i}\bigg( \frac{a^{\prime }(\theta ^{\prime })}{2\sqrt{\lambda-\widetilde{a}}i}-W^{2}(\theta ^{\prime })\bigg) d\theta
^{\prime }
\\
&\qquad+ie^{-2\sqrt{\lambda-\widetilde{a}}\theta
i}\int_{0}^{\theta }e^{2\sqrt{\lambda-\widetilde{a}}\theta ^{\prime }i}
\left[ \widetilde{a}-a(\theta ^{\prime })-W^{2}(\theta ^{\prime
})\right] d\theta ^{\prime }  \label{eqw} \\
& =
e^{-2\sqrt{\lambda-\widetilde{a}}\theta i}
\frac{\widetilde{a} -a(0)}{2\sqrt{\lambda-\widetilde{a}}
}\\
&\qquad+\frac{ie^{-2\sqrt{\lambda-\widetilde{a}}(\theta+2\pi) i}}{1-e^{-4\sqrt{\lambda-
\widetilde{a}}\pi i}}\int_{0}^{2\pi }e^{2\sqrt{\lambda-\widetilde{a}}
\theta ^{\prime }i}\bigg( \frac{a^{\prime }(\theta ^{\prime })}{2\sqrt{\lambda-\widetilde{a}}i}-W^{2}(\theta ^{\prime })\bigg) d\theta
^{\prime }\\
&\qquad-\frac{
a(\theta )-e^{-2\sqrt{\lambda-\widetilde{a}}\theta i}a(0)}{2\sqrt{\lambda-
\widetilde{a}}}+\frac{\widetilde{a}-e^{-2\sqrt{\lambda-\widetilde{a}}\theta
i}\widetilde{a}}{2\sqrt{\lambda-\widetilde{a}}}  \notag \\
& \qquad +ie^{-2\sqrt{\lambda-\widetilde{a}}\theta i}\int_{0}^{\theta }e^{2
\sqrt{\lambda-\widetilde{a}}\theta ^{\prime }i}\left[ \frac{a^{\prime
}(\theta ^{\prime })}{2\sqrt{\lambda-\widetilde{a}}i}-W^{2}(\theta ^{\prime })\right] d\theta ^{\prime }.
\end{align*}
Moreover the map $\lambda\mapsto W_\lambda$ is continuous as a map
from $I_\delta$ to $C^0({\mathbb S}^1)$.
\end{Lemma}
\begin{pf}
  It is easy to verify that there exist $\bar\lambda_\delta$ and  $C_\delta>0$ such that for every
$\lambda\in I_\delta$, $\lambda\geq\bar\lambda_\delta$, $T_\lambda$
maps $\overline{B}_{C_\delta/{\sqrt{\lambda-\widetilde{a}}}}=\{u\in
C^0({\mathbb S}^1):\ \sup_{{\mathbb S}^1}|u|\leq
C_\delta/{\sqrt{\lambda-\widetilde{a}}}\}$ into itself and is a
contraction there. The conclusion then follows from the Banach
contraction mapping theorem.
\end{pf}

\noindent For $\lambda\in I_\delta$, $\lambda\geq\bar\lambda_\delta$, let
$W_\lambda$ be as in Lemma \ref{l:contraction1}. Then it is easy to
verify that $W_\lambda$
satisfies
\begin{equation}\label{eq:wlameq}
\begin{cases}
  -i W_\lambda'(\theta)+2\sqrt{\lambda-\widetilde
    a}\,W_\lambda(\theta)+W_\lambda^2(\theta)=\widetilde
  a-a(\theta),&\text{in }[0,2\pi],\\
W_\lambda(0)= W_\lambda(2\pi).&
\end{cases}
\end{equation}
Letting
\begin{equation}\label{eq:slamb}
S_\lambda(\theta):=\sqrt{\lambda-\widetilde
    a}\,\theta+\int_0^\theta W_\lambda(\theta')\,d\theta',
\end{equation}
we have that $S_\lambda$ satisfies
\begin{equation}\label{eq:slambda}
\begin{cases}
  -i S_\lambda''(\theta)+(S_\lambda'(\theta))^2=\lambda-a(\theta),&\text{in }[0,2\pi],\\
S_\lambda'(0)= S_\lambda'(2\pi),\\
S_\lambda(0)=0.
\end{cases}
\end{equation}
\begin{Lemma}\label{l:mediaW}
If $\lambda$ is sufficiently large, then $\int_0^{2\pi} W_\lambda(\theta)\,d\theta\in\R$.
\end{Lemma}
\begin{pf}
Let us define $\eta_\lambda(\theta)=\Re S_\lambda(\theta)$
and
 $\xi_\lambda(\theta)=\Im S_\lambda(\theta)$. Then
 \eqref{eq:slambda} implies that
\begin{equation*}
-\eta_\lambda^{\prime \prime }+2\eta_\lambda^{\prime }\xi_\lambda
^{\prime }=0,\quad\text{in }[0,2\pi],
\end{equation*}
so that
\begin{equation}\label{eq:eta'}
\eta_\lambda ^{\prime }(\theta )=C_\lambda e^{2\xi_\lambda (\theta )},\quad\text{in }[0,2\pi],
\end{equation}
where $C_\lambda=\sqrt{\lambda-\widetilde
    a}+\Re W_\lambda(0)$. We notice that
  \eqref{eq:inteball} implies that, if $\lambda$ is sufficiently
  large, then $C_\lambda\neq 0$. The condition $S_\lambda'(0)=
  S_\lambda'(2\pi)$ implies  that $\eta_\lambda ^{\prime }(0)=
  \eta_\lambda ^{\prime }(2\pi)$ and hence from \eqref{eq:eta'} it
  follows that
\begin{equation*}
\xi_\lambda (0)=\xi_\lambda  (2\pi ).
\end{equation*}
Since $ S_\lambda(0)=0$, we have that $\xi_\lambda (0)=0$ and then
\begin{align*}
\frac1{2\pi}\int_0^{2\pi} W_\lambda(\theta)\,d\theta&=-\sqrt{\lambda-\widetilde
    a}+\frac1{2\pi}(\eta_\lambda(2\pi)+i\xi_\lambda(2\pi))\\
&=
-\sqrt{\lambda-\widetilde
    a}+\frac{\eta_\lambda(2\pi)}{2\pi}\in\R.
\end{align*}
\end{pf}

\begin{Lemma}\label{l:flambda}
Let $\bar A\in \R$ such that $\bar A\not\in \frac12 \Z$
and let $0<\delta< \mathop{\rm dist}(\bar A,\tfrac1
{2}\Z)$. Then there exists $\bar k\in\N$ such that for all $k\in \N$,
$k\geq\bar k$, there exist
 $\lambda_k^+,\lambda_k^-\in I_\delta$ such that
$\lambda_k^+\geq  \bar\lambda_\delta$, $\lambda_k^-\geq  \bar\lambda_\delta$ and
\begin{align}
\label{eq:2}&\sqrt{\lambda_k^+-\widetilde a}=\bar
A-\frac1{2\pi}\int_0^{2\pi}W_{\lambda_k^+}(\theta)\,d\theta +k,\\
\label{eq:3}&\sqrt{\lambda_k^--\widetilde a}=-\bar
A-\frac1{2\pi}\int_0^{2\pi}W_{\lambda_k^-}(\theta)\,d\theta +k.
\end{align}
\end{Lemma}
\begin{pf}
Let $g:[\bar\lambda_\delta,+\infty)\to\R$ be a continuous function such
that
\[
g(\lambda)=
\frac1{2\pi}\int_0^{2\pi}W_{\lambda}(\theta)\,d\theta\quad\text{for
  all }
\lambda\in I_\delta\
\]
 and $|g(\lambda)|\leq C_\delta/\sqrt{\lambda-\tilde
  a}$ for all $\lambda\geq \bar\lambda_\delta$. Then the function
\[
f:[\bar\lambda_\delta,+\infty)\to\R,\quad
f(\lambda)=\sqrt{\lambda-\tilde a}-\bar A+g(\lambda),
\]
is continuous and
$\lim_{\lambda\to+\infty}f(\lambda)=+\infty$. Therefore there exists
$\bar k$ sufficiently large such that, for all $k\geq \bar k$, there exists
 $\lambda_k^+\geq  \bar\lambda_\delta$ such that $f(\lambda_k^+)=k$,
 i.e.
\begin{equation}\label{eq:1}
\sqrt{\lambda_k^+-\tilde a}=k+\bar A-g(\lambda_k^+).
\end{equation}
If $\bar k$ is sufficiently large,
\eqref{eq:1} implies that
\[
\mathop{\rm dist}\Big(\sqrt{\lambda_k^+-\tilde
  a},\tfrac1 {2}\Z\Big)=\mathop{\rm dist}(\bar A,\tfrac1 {2}\Z)-g(\lambda_k^+)>\delta,
\]
so that $\lambda_k^+\in I_\delta$ and \eqref{eq:2} is proved. The
proof of \eqref{eq:3} is analogous.
\end{pf}

\begin{Lemma}\label{l:flambdaasy}
Under the same assumptions as in Lemma \ref{l:flambda}, let, for all
$k\geq\bar k$,
 $\lambda_k^+,\lambda_k^-\in I_\delta$  as in Lemma
 \ref{l:flambda}. Then
\begin{equation}\label{eq:intwl}
\int_0^{2\pi}W_{\lambda_k^\pm}(\theta)\,d\theta=O \big(\tfrac1{k^3}\big),\quad\text{as
}k\to+\infty
\end{equation}
and
\begin{align}
\label{eq:5asy}&\sqrt{\lambda_k^+-\widetilde a}=\bar
A +k+O(\tfrac1{k^3}),\quad \sqrt{\lambda_k^--\widetilde a}=-\bar
A+k+O(\tfrac1{k^3}),\\
\label{eq:6asy}&\lambda_k^+=\widetilde a+(\bar
A +k)^2+O(\tfrac1{k^2}),\quad \lambda_k^-=\widetilde a+(-\bar
A+k)^2+O(\tfrac1{k^2}),
\end{align}
as $k\to+\infty$.
\end{Lemma}
\begin{pf}
By integrating \eqref{eq:wlameq} between $0$ and $2\pi$ and using
estimate \eqref{eq:inteball} we have that
\begin{equation*}
\bigg|  \int_0^{2\pi}W_{\lambda}(\theta)\,d\theta \bigg|=\bigg|-\frac{1}{2\sqrt{\lambda-\widetilde
    a}}\int_0^{2\pi}W_{\lambda}^2(\theta)\,d\theta \bigg|\leq
\frac{\pi C_\delta}{(\lambda-\widetilde a)^{3/2}}.
\end{equation*}
Since from \eqref{eq:2} and \eqref{eq:3} it follows that
$\lambda_k^\pm\sim k^2$ as $k\to+\infty$, we derive \eqref{eq:intwl},
which yields \eqref{eq:5asy} (and the \eqref{eq:6asy} by squaring) in view of \eqref{eq:2} and \eqref{eq:3}.
\end{pf}

\begin{Lemma}\label{l:insp}
Let $a\in W^{1,\infty}({\mathbb S}^1)$ and  $\bar A\in\R\setminus \frac12\Z$.
 If  $k\geq\bar k$, then
\[ \lambda_k^+, \lambda_k^-\in \sigma(\frak L_{\bar A,a}),
\]
where $\frak  L_{\bar A,a}\varphi=
-(\varphi)''+[a(\theta)+\bar A^2] \varphi-2i\bar A
\varphi$.
Moreover
\begin{equation}\label{eq:4}
\varphi_k^+(\theta)=e^{-i\bar A\theta}e^{iS_{\lambda_k^+}(\theta)},\quad
   \varphi_k^-(\theta)=e^{-i\bar A\theta}e^{-i\overline{S_{\lambda_k^-}(\theta)}},
\end{equation}
are eigenfunctions of $\frak L_{\bar A,a}$ associated to $\lambda_k^+,
\lambda_k^-$ respectively.
\end{Lemma}
\begin{pf}
By direct calculations, we have that $\varphi_k^{\pm}$ satisfy
\begin{align*}
-(\varphi_k^\pm)''(\theta)+[a(\theta)+\bar A^2]
\varphi_k^\pm(\theta)-2i\bar A
(\varphi_k^\pm)'(\theta)=\lambda_k^\pm\varphi_k^\pm(\theta)
\end{align*}
in $[0,2\pi]$, i.e. $\varphi_k^{\pm}$ are non-trivial solutions to
\[
\frak L_{\bar A,a}\varphi_k^\pm=\lambda_k^\pm\varphi_k^\pm\quad\text{in }[0,2\pi].
\]
Furthermore \eqref{eq:2} and \eqref{eq:3} imply that
\begin{align*}
 \varphi_k^\pm(0)= \varphi_k^\pm (2\pi)\quad\text{and}\quad (\varphi_k^\pm)'(0)= (\varphi_k^\pm)' (2\pi).
\end{align*}
The lemma is thereby proved.
\end{pf}

\begin{Lemma}\label{l:usekato}
Let $a\in W^{1,\infty}({\mathbb S}^1)$ and $\bar A\in(-\frac12,\frac12)\setminus\{0\}$.
For every $\bar\alpha\geq \|a+\bar A^2\|_{L^\infty}^2$ and $c\geq 0$,
there exist $\bar\lambda>0$ and  $k_0>\bar k$ such that
 \[
\sigma(\frak L_{\bar A,a})\cap[\bar\lambda,+\infty)\subset\bigcup_{k=k_0}^\infty
B\big(k^2,c+\sqrt{\bar\alpha +4k^2\bar A^2}\,\big).
\]
Furthermore, if $k\geq k_0$, each ball $B\big(k^2,\sqrt{\bar\alpha
  +4k^2{\bar A^2}}\big)$ contains exactly two
eigenvalues of $\frak L_{\bar A,a}$ (counted with their own multiplicity).
\end{Lemma}
\begin{pf}
We apply lemma \ref{lem:kato} with
\begin{equation*}
  L_0:=-\frac{d^2}{d\theta^2}
  \qquad
  L:=-\frac{d^2}{d\theta^2}-2
  i\bar A\frac{d}{d\theta}
  +\alpha(\theta),
\end{equation*}
where $\alpha(\theta)=a(\theta)+\bar A^2$. Let
\begin{equation*}
  R_0(\lambda)=\Big(-\frac{d^2}{d\theta^2}-\lambda I\Big)^{-1},
  \quad
  R(\lambda)=\Big(-\frac{d^2}{d\theta}-2i\bar A\frac{d}{d\theta}
  +\alpha(\theta)-\lambda I\Big)^{-1}.
\end{equation*}
Notice that, via Fourier we can write, for $f=\sum(\alpha_k\sin(k\theta)+\beta_k\cos(k\theta))$,
\begin{equation*}
  R_0(\lambda)f=
  \sum
  \frac{1}{k^2-\lambda}(\alpha_k\sin(k\theta)+\beta_k\cos(k\theta)),
\end{equation*}
therefore we have the estimate
\begin{equation*}
  \|R_0\|_{\mathcal L(L^2)}\leq\frac{1}{\text{dist}(\lambda,\{k^2:k\in\Z\})}.
\end{equation*}
Now notice that, formally, we can write
\begin{equation}\label{eq:2b}
  R(\lambda)=R_0(\lambda)\left(I+WR_0(\lambda)\right)^{-1},
\end{equation}
being $W=-2i\bar A\frac{d}{d\theta}+\alpha(\theta)$ a first order operator.
Since $\frac{d}{d\theta}$ commutes with $R_0$ (by spectral theorem), we can write as follows:
\begin{multline*}
  WR_0(\lambda)f=-2i\bar A\sum_k\frac{k(\alpha_k\cos(k\theta)-\beta_k\sin(k\theta))}{k^2-\lambda}\\
  +\alpha(\theta)\sum_k
  \frac{1}{k^2-\lambda}(\alpha_k\sin(k\theta)+\beta_k\cos(k\theta)).
\end{multline*}
We hence obtain the following: if $\Re\lambda$ is large enough,
$\bar\alpha\geq \|\alpha\|_{L^\infty}^2$, $c\geq 0$,  and
\begin{equation*}
  |\lambda-k^2|>c+\left(4k^2\bar A^2+\bar\alpha\right)^{\frac12},
\end{equation*}
then
\begin{equation}\label{eq:3b}
  \|WR_0\|_{\mathcal L(L^2)}<1.
\end{equation}
Then, by \eqref{eq:2b} and \eqref{eq:3b}, if $k$ is large enough,
outside of any ball with center in $k^2$ and radius
$c+\left(4k^2\bar A^2+\bar\alpha\right)^{\frac12}$,
the operator $R(\lambda)$ is bounded for large $\lambda$, hence we do
not have large eigenvalues outside that balls. We notice that, if $k$
is large, the balls with center in $k^2$ and radius
$c+\left(4k^2\bar A^2+\bar\alpha\right)^{\frac12}$
are mutually disjoint, since $|\bar A|<\frac12$
implies that
\[
2c+\left(4k^2\bar A^2+\bar\alpha\right)^{\frac12}+
\left(4(k+1)^2\bar A^2+\bar\alpha\right)^{\frac12}<(k+1)^2-k^2
\]
provided that $k$ is sufficiently large.
On the other hand, if $\Gamma$ is the circle with center in $k^2$ and
radius
$\left(4k^2\bar A^2+\bar\alpha\right)^{\frac12}$
with $k$ large, we can easily estimate
\begin{equation*}
  \left\|\frac{1}{2\pi i}\int_\Gamma(R(\lambda)-R_0(\lambda))f\,d\lambda\right\|_{L^2}
  <\|f\|_{L^2},
\end{equation*}
(use the Born expansion $(I+WR_0)^{-1}=\sum(WR_0)^n$)
which together with point (2) of Lemma \ref{lem:kato} gives the desired result.

Therefore, outside those balls there are no eigenvalues, and inside
there are the same number of eigenvalues both for $L_0$ and $L$: this
number is 2.
\end{pf}

\begin{Lemma}\label{l:lambdainball}
Let $a\in W^{1,\infty}({\mathbb S}^1)$ and $\bar
A\in(-\frac12,\frac12)\setminus\{0\}$.
 Let $\lambda_k^\pm$ be as in Lemma \ref{l:flambda}. Then there exist
 $c>0$, $\bar\alpha>0$,
$\bar\lambda>0$ and $\tilde k$ such that
\begin{itemize}
\item[(i)]
 for all $k\geq \tilde k$,
$\lambda_k^+,\lambda_k^-\in B\Big(k^2,c+\sqrt{\bar\alpha +4k^2\bar A^2}\,\Big)$;
\item [(ii)]
$\sigma(\frak L_{\bar A,a})\cap[\tilde\lambda,+\infty)=\{\lambda_k^+,\lambda_k^-:k\geq
\tilde k\}$.
\end{itemize}
\end{Lemma}
\begin{pf}
From \eqref{eq:6asy} we have that, if $c,\bar\alpha>0$ are chosen
sufficiently large,
\begin{equation*}
|\lambda_k^+-k^2|=\Big|\widetilde a+
\bar A^2+2k\bar A+O(\tfrac1{k^2})\Big|<c+\sqrt{\bar\alpha +4k^2\bar
  A^2}
\end{equation*}
if $k$ is large enough, thus proving
(i) for $\lambda_k^+$. The proof of (i) for
$\lambda_k^-$ is analogous.

The statement (ii) follows by combining (i) and Lemma \ref{l:usekato}.
\end{pf}

\begin{proof}[Proof of Lemma \ref{lem:spectral}]
  \noindent From Lemmas \ref{l:probequiv} and \ref{l:lambdainball} it
  follows that there exist $k^*\in \N$ and $\ell\in\Z$ such that
$\{\mu_k:\ k>
k^*\}=\{\lambda_j:\ j\in\Z,\ |j|\geq \ell\}$ where
\[
\lambda_j=
\begin{cases}
  \lambda^-_{|j|},&\text{if }j<0,\\
  \lambda^+_{|j|},&\text{if }j>0.
\end{cases}
\]
Then, in view of Lemma \ref{l:flambdaasy}
\[
\sqrt{\lambda_j-\widetilde a}=(\sgn j)\bar A+|j|+O\big(\tfrac1{|j|^3}\big),\quad\text{as }|j|\to+\infty.
\]
From \eqref{eq:4}, \eqref{eq:slamb}, \eqref{eq:intwl}, and \eqref{eq:5asy},
it follows that
\begin{align*}
&\varphi_j^+(\theta)=e^{-i\bar A\theta}\Big(e^{i(\bar A+j)\theta}+O\big(\tfrac1{|j|^3}\big)\Big),\quad
\text{as }j\to+\infty,\\
&\varphi_j^-(\theta)=e^{-i\bar A\theta}\Big(e^{i(\bar A-j)\theta}+O\big(\tfrac1{|j|^3}\big)\Big),\quad
\text{as }j\to+\infty.
\end{align*}
Therefore, letting, for $j\in\Z$ such that $|j|\geq \ell$,
\begin{align*}
\widetilde \phi_j=
\begin{cases}
\frac{\varphi^-_{|j|}}{\|\varphi^-_{|j|}\|_{L^2(0,2\pi)}},&\text{ if }j<0,\\
\frac{\varphi^+_{|j|}}{\|\varphi^+_{|j|}\|_{L^2(0,2\pi)}},&\text{ if }j>0,
\end{cases}
\end{align*}
we have that, for $|j|\geq \ell$, $\widetilde\phi_j$ is a
$L^{2}\big((0,2\pi),{\mathbb{C}}\big)$-normalized
eigenfunction of the operator $\frak L_{\bar A,a}$
corresponding to the eigenvalue $\lambda_{j}$ and
\begin{align*}
  \widetilde\phi_j(\theta)=
\frac1{\sqrt{2\pi}}e^{-i\bar A\theta}\Big(e^{i(\bar A+j)\theta}+R_j(\theta)\Big),
\end{align*}
where $\|R_j\|_{L^\infty(0,2\pi)}= O\big(\tfrac1{|j|^3}\big)$ as
$j\to\infty$. Hence, in view of Lemma \ref{l:probequiv} we have that
$\phi_j(\cos\theta,\sin\theta)=e^{i\bar A\theta}e^{-i\int_0^\theta
  A(s)\,ds}\widetilde\phi_j(\theta)$ is a
$L^{2}\big({\mathbb{S}}^{1},{\mathbb{C}}\big)$-normalized
eigenfunction of the operator $L_{{\mathbf{A}},a}$ on $\mathbb{S}^{1}$
corresponding to the eigenvalue $\lambda_{j}$ and
\begin{align*}
  \phi_j(\cos\theta,\sin\theta)=
\frac1{\sqrt{2\pi}}e^{-i\big(\lfloor\widetilde A+1/2\rfloor \theta+\int_0^\theta
  A(t)\,dt\big)}\Big(e^{i(\widetilde A+j)\theta}+R_j(\theta)\Big).
\end{align*}
The proof is thereby complete.
\end{proof}

By means of the previous result, we immediately obtain the following Corollary.

\begin{Corollary}
\label{cor:spectral} Let $k^*,\ell$ as
in Lemma \ref{lem:spectral} and $K$ be given by \eqref{nucleo},
 with  $\psi_k$
being any $L^{2}\big({\mathbb{S}}^{1},{\mathbb{C}}\big)$-normalized
eigenfunctions of $L_{{\mathbf{A}},a}$ on $\mathbb{S} ^{1}$ if $k\leq
k^*$ and $\psi_k=\phi_j$ if $k>k^*$ and $\mu_k=\lambda_j$, with
$\lambda_j,\phi_j$ being as in Lemma \ref{lem:spectral}. Then, we have that
\begin{align}  \label{eq:Kgauge}
&K(x,y) =
\sum\limits_{k=1}^{k^* }i^{-\beta _{k}}j_{-\alpha
_{k}}(r r')\psi _{k}(\theta)\overline{\psi _{k}(\theta')}\\
\notag&+
\frac1{2\pi}e^{-i\int_{\theta^{\prime }}^\theta
A(s)\,ds}e^{-i[\widetilde A+\frac12](\theta-\theta')}\\
&\notag\qquad\times\sum\limits_{|j|\geq\ell}i^{-\beta(\lambda_j)}j_{-\alpha(\lambda_j)
}(rr')
\Big(e^{i(\widetilde A+j)\theta}+R_j(\theta)\Big)
\Big(e^{-i(\widetilde A+j)\theta'}+\overline{R_j(\theta')}\Big),
\end{align}
if
$x=(r\cos\theta,r\sin\theta)$ and $y=(r'\cos\theta',r'\sin\theta')$,
where
\begin{equation}\label{eq:alphalamj}
\alpha(\lambda_j):=-\sqrt{\lambda_j}, \quad \beta(\lambda_j) :=\sqrt{\lambda_j},
\end{equation}
and $R_j$ is as in  Lemma~\ref{lem:spectral}.
\end{Corollary}

\section{Proof of the main result}

\label{sec:proof}

We can now perform the proof of Theorem \ref{thm:main}. Let us first
assume that condition \eqref{eq:usek} holds, so that the asymptotic
expansion
of eigenvalues and eigenfunctions stated in Lemma \ref{lem:spectral}
holds. Let $K$ be defined
by \eqref{nucleo}; by formula \eqref{eq:representation}, it is sufficient to
show that
\begin{equation*}
\sup_{x,y\in\mathbb{R}^2}\left|K(x,y)\right|<\infty.
\end{equation*}
In particular, the study of the boundedness of $K$ is reduced, thanks to
Corollary \ref{cor:spectral}, to the study of the boundedness of the
two
series
\begin{equation}\label{eq:1ser}
\Sigma_{k\leq k^*}=\sum\limits_{k=1}^{k^* }i^{-\beta _{k}}j_{-\alpha
_{k}}(r r')\psi _{k}(\theta)\overline{\psi _{k}(\theta')},
\end{equation}
and
\begin{equation}\label{eq:2ser}
\Sigma_{|j|\geq\ell}=\sum\limits_{|j|\geq\ell}i^{-\beta(\lambda_j)}j_{-\alpha(\lambda_j)
}(rr')
\Big(e^{i(\widetilde A+j)\theta}+R_j(\theta)\Big)
\Big(e^{-i(\widetilde A+j)\theta'}+\overline{R_j(\theta')}\Big)
\end{equation}
uniformly with respect to $r,r',\theta,\theta'$.
Since $\mu_1(\mathbf{A},a)>0$,  all the indices $\alpha_k$ in \eqref{eq:alfabeta} are negative.
Therefore, the Bessel functions $j_{-\alpha_k}$ are bounded functions, for
any $k$. In addition, the functions $\psi_k$ are obviously bounded, for any $k$: as a consequence, we obtain that
\begin{equation}  \label{eq:finita}
\sup_{\substack{r,r'\geq0\\\theta,\theta'\in{\mathbb S}^1}}\left|\Sigma_{k\leq k^*}(r,r',\theta,\theta')\right| <\infty.
\end{equation}
In order to prove that $\Sigma_{|j|\geq\ell}$ is uniformly bounded, we compare it with the analogous
kernel $K_{ab}$ associated to the Aharonov-Bohm potential $\mathbf{A}
_{ab}:=\alpha\big(-\frac{x_2}{|x|^2},\frac{x_1}{|x|^2}\big)$, with $\alpha\in\mathbb{R}$,
given by
\begin{equation*}
K_{ab}(x,y)=
\sum\limits_{k\in\Z}i^{-\beta _{k}^{ab}}j_{-\alpha
_{k}^{ab}}(|x||y|)\psi _{k}^{ab}\big(\tfrac{x}{|x|}\big)\overline{\psi _{k}^{ab}\big(\tfrac{
y}{|y|}\big)},
\end{equation*}
where $\psi_k^{ab}$ are the eigenfunctions  defined in
\eqref{eq:eigfab2}  of $L_{{\mathbf A}_{ab},0}$
associated to the eigenvalue $\mu_k^{ab}=(k+\alpha)^2$, and $\alpha_k^{ab}, \beta_k^{ab}$ are given by
\eqref{eq:alfabeta} with $\mu_k$ replaced by $\mu_k^{ab}$. We have explicitly
\begin{equation*}
\alpha_k^{ab} =-\sqrt{\mu_k^{ab}}=-\left|k+\alpha\right|, \qquad
\beta_k^{ab} =\sqrt{\mu_k^{ab}}=\left|k+\alpha\right|.
\end{equation*}
We choose $\alpha=\bar A$ with $\bar A$ as in \eqref{eq:defAbar}, denote
\[
\Sigma^{ab}_{|j|\geq \ell}(r,r',\theta,\theta')=\sum\limits_{|j|\geq\ell}
i^{-|j+\alpha|}j_{|j+\alpha|}(r r')e^{ij\theta}e^{-ij\theta'},
\]
and write
\begin{equation}\label{eq:decab}
\Sigma_{|j|\geq \ell}=
  \Big(\Sigma_{|j|\geq \ell}- e^{i\bar A(\theta-\theta')}\Sigma^{ab}_{|j|\geq \ell}\Big)+
  e^{i\bar A(\theta-\theta')} \Sigma^{ab}_{|j|\geq \ell}.
\end{equation}
In the paper \cite{FFFP} it has been shown that
\begin{equation}  \label{eq:abest}
\sup_{\substack{r,r'\geq0\\\theta,\theta'\in{\mathbb S}^1} }\left|e^{i\bar A(\theta-\theta')}\Sigma^{ab}_{|j|\geq\ell}(r,r',\theta,\theta')\right|=
\sup_{\substack{r,r'\geq0\\\theta,\theta'\in{\mathbb S}^1} }\left|\Sigma^{ab}_{|j|\geq\ell}(r,r',\theta,\theta')\right|<\infty.
\end{equation}
To prove the uniform bound of $\Sigma_{|j|\geq \ell}$  is hence sufficient to prove the following claim:
\begin{equation} \label{eq:claim}
  \sup_{\substack{r,r'\geq0\\\theta,\theta'\in{\mathbb S}^1}
  }\left|\Sigma_{|j|\geq \ell}(r,r',\theta,\theta')-e^{i\bar A(\theta-\theta')}
    \Sigma^{ab}_{|j|\geq \ell}(r,r',\theta,\theta')\right|<\infty.
\end{equation}
In view of the above considerations, we now pass to prove that
\eqref{eq:claim} holds.

Let us write
\begin{equation}  \label{eq:decomp}
\Sigma_{|j|\geq \ell}-e^{i\bar A(\theta-\theta')}\Sigma^{ab}_{|j|\geq \ell}=K_{1}+K_{2},
\end{equation}
where
\begin{align*}
K_{1} & =\sum_{|j| \geq \ell}\left[ i^{-\beta(\lambda_j)}J_{-\alpha(\lambda_j)}(r r')-i^{-\left|j+\bar A\right|}
J_{\left|j+\bar A\right|}(r r')
\right] e^{i(j+\bar A)\theta }e^{-i(j+\bar A)\theta ^{\prime }} \\
K_{2} & =\sum_{|j|\geq \ell}i^{-\beta(\lambda_j)}J_{-\alpha(\lambda_j)}(r r') \times\\
&\hskip1cm\times\left[
\Big(e^{i(\bar A+j)\theta}\!\!+R_j(\theta)\Big)
\Big(e^{-i(\bar A+j)\theta'}\!\!+\overline{R_j(\theta')}\Big)-e^{i(j+\bar A)\theta
}e^{-i(j+\bar A)\theta ^{\prime }}\right].
\end{align*}
Here we used the fact that in dimension $N=2$ we have $j_s\equiv J_s$, for
any $s\in\mathbb{R}$.

Let us now recall the estimate
\begin{equation}  \label{est_vega}
\left\vert J_{\nu }(r)\right\vert \leq \frac{C}{\left\vert \nu \right\vert
^{ \frac{1}{3}}}
\end{equation}
(see e.g. \cite{BRV,landau}), which holds for some $C>0$ independent
of $x$ and $\nu$. Moreover, by \eqref{eq:eigenvalues} and \eqref{eq:alphalamj} we have that
\begin{equation}  \label{eq:asyeig}
-\alpha(\lambda_j)\sim |j| \qquad \text{as }|j|\to\infty.
\end{equation}
In addition by Lemma \ref{lem:spectral}
\begin{multline}  \label{est_autof}
\left\|\Big(e^{i(\bar A+j)\theta}\!\!+R_j(\theta)\Big)
\Big(e^{-i(\bar A+j)\theta'}\!\!+\overline{R_j(\theta')}\Big)-e^{i(j+\bar A)\theta
}e^{-i(j+\bar A)\theta ^{\prime }}\right\|_{L^\infty({\mathbb
  S}^1)}\\
=O\Big(\frac1{|j|^3}\Big)
\end{multline}
as $|j|\to+\infty$. Hence, by \eqref{est_vega}, \eqref{eq:asyeig} and
\eqref{est_autof} one easily gets
\begin{equation}  \label{eq:k2}
\sup_{\substack{r,r'\geq0\\\theta,\theta'\in{\mathbb S}^1}
  }\left|K_2(r,r',\theta,\theta')\right|\leq C\sum_{|j|\geq\ell}|j|^{-\frac{10}3}<\infty.
\end{equation}
In order to get the analogous estimate for $K_1$, we now introduce another
well known representation formula for the Bessel functions. Let $
\gamma\subset\mathbb{C}$ be the positively oriented contour represented in
Figure \ref{fig:pspic}.

\begin{figure}[h]
\begin{pspicture}(-4,-2.5)(4,2.5)
\psset{arrowscale=1.7}
\psline[linewidth=0.02cm](-4,0)(4,0)
\psline[linewidth=0.02cm](0,-2.5)(0,2.5)
\psline[linewidth=0.04cm]{->}(-4,-0.2)(-0.95,-0.2)
\psline[linewidth=0.04cm]{->}(-4,-0.2)(-2,-0.2)
\psline[linewidth=0.04cm]{->}(-4,-0.2)(-3,-0.2)
\psline[linewidth=0.04cm]{<-}(-4,0.2)(-1,0.2)
\psline[linewidth=0.04cm]{<-}(-3,0.2)(-1,0.2)
\psline[linewidth=0.04cm]{<-}(-2,0.2)(-1,0.2)
 \usefont{T1}{ptm}{m}{n}
 \rput(1,1){{$\bf\Gamma_1$}}
 \usefont{T1}{ptm}{m}{n}
 \rput(-2.5,0.5){{$\bf\Gamma_2$}}
\psarc[linewidth=0.04]{->}(0,0){1.001998}{190.6}{169.6900423}
\psarc[linewidth=0.04]{->}(0,0){1.001998}{190.6}{90}
\psarc[linewidth=0.04]{->}(0,0){1.001998}{190.6}{0}
\psarc[linewidth=0.04]{->}(0,0){1.001998}{190.6}{-90}
\end{pspicture}
\caption{Integration oriented domain $\protect\gamma$.}
\label{fig:pspic}
\end{figure}
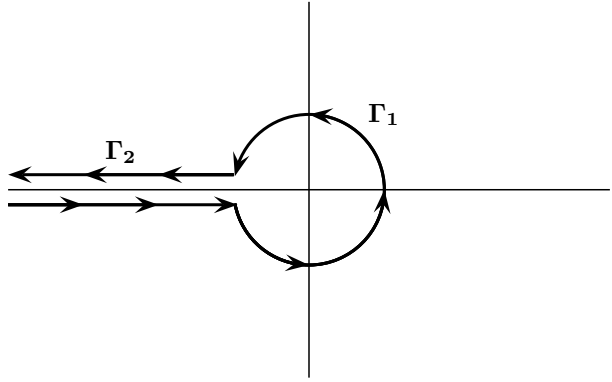

Then we have the representation
\begin{equation*}
J_{\nu }(r)=\frac{1}{2\pi i}\int_{\gamma }e^{\frac{r}{2}\left( z-\frac{1}{z}
\right) }\frac{dz}{z^{\nu +1}}
\end{equation*}
(see \cite[5.10.7]{lebedev}). Consequently, we obtain
\begin{align}  \label{eq:nuova}
&K_1(r,r',\theta,\theta')\\
&   = \frac{1}{2\pi i}\sum_{|j|\geq \ell} \int_\gamma \frac{1}{z}e^{\frac{rr'}{2}
\left(z-\frac
  1z\right)}\left[(iz)^{\alpha(\lambda_j)}-(iz)^{-\left|j+\bar A\right|}
\right] e^{i(j+\bar A)(\theta-\theta^{\prime})}\,dz  \notag \\
& = \frac{1}{2\pi i}\sum_{|j|\geq \ell} \int_\gamma \frac{1}{z}e^{\frac{rr'}{2}
\left(z-\frac 1z\right)}(iz)^{-\left|j+\bar A\right|} \left[(iz)^{-\sqrt{
\lambda_j}+\left|j+\bar A\right|}-1\right] e^{i(j+\bar A)(\theta-\theta^{\prime})}\,dz.
\notag
\end{align}
From
\eqref{eq:eigenvalues} it follows that
\begin{equation}\label{eq:7}
-\sqrt{
\lambda_j}+\left|j+\bar A\right|
= \sqrt{(j+\bar A)^2}- \sqrt{
\widetilde a+(j+\bar A)^2+O\big(\tfrac1{j^2}\big)} = -\frac{\widetilde a}{2|j|}+O\left(j^{-2}\right).
\end{equation}
Therefore, a first-order Taylor expansion in the last term of
\eqref{eq:nuova} gives in turn
\begin{multline}  \label{eq:nuova2}
K_1(r,r',\theta,\theta') \\= \frac{1}{2\pi i}
\sum_{|j|\geq \ell} \int_\gamma \frac{1}{z}e^{\frac{rr'}{2}
\left(z-\frac 1z\right)}\left[-\frac{\widetilde a\log(iz)}{2|j|} \cdot \frac{
e^{i(j+\bar A)(\theta-\theta^{\prime})}}{(iz)^{|j+\bar A|}}+\mathcal R_j(z)\right]
\,dz
\end{multline}
where $\|\mathcal R_j(z)\|_{L^\infty(\gamma)}=O(j^{-2})$ as $|j|\to+\infty$.

We observe that it is possible to exchange the order of summation and integration in
\eqref{eq:nuova2}, see the proof of Theorem 1.11 in \cite{FFFP} for details. We hence get
\begin{multline}  \label{eq:newchange}
K_1(r,r',\theta,\theta')\\
 = -\frac{1}{2\pi i}\int_\gamma \frac{1}{2z}e^{\frac {rr'}2\left(z-\frac
1z\right)} \widetilde{a}\log(iz) \sum_{|j|\geq \ell}\left[\frac1{|j|}
\frac{e^{i(j+\bar A)(\theta-\theta^{\prime})}}{(iz)^{|j+\bar A|}}+O\left(j^{-2}\right)\right]
\,dz.
\end{multline}
Finally, we notice that (if $\ell$ is large enough)
\begin{align*}
&\sum_{|j|\geq\ell}^{\infty }
\frac1{|j|}
\frac{e^{i(j+\bar A)(\theta-\theta^{\prime})}}{(iz)^{|j+\bar A|}}\\
&= -\frac{e^{i\bar A(\theta-\theta')}}{(iz)^{\bar A}}\log \left[
  1-\frac{e^{i(\theta -\theta ^{\prime })}}{iz}\right]
 -\frac{e^{i\bar A(\theta-\theta')}}{(iz)^{-\bar A}}\log \left[
   1-\frac{e^{-i(\theta -\theta ^{\prime })}}{iz}\right]\\
&\quad+\sum_{1\leq|j|<\ell}\frac{1}{|j|}
\frac{e^{i(j+\bar A)(\theta-\theta^{\prime})}}{(iz)^{|j|+\bar A\sgn j}}
,
\end{align*}
which together with \eqref{eq:newchange} leads to
\begin{align*}
&K_1(r,r',\theta,\theta')\\
 \notag&= \frac{e^{i\bar A(\theta-\theta')}}{2\pi i}\int_\gamma \frac{1}{2z}e^{\frac {rr'}2\left(z-\frac
1z\right)} \widetilde{a}\log(iz) \bigg(\frac{\log \big(
  1-\tfrac{e^{i(\theta -\theta ^{\prime })}}{iz}\big)}{(iz)^{\bar A}}
 +\frac{\log \big(
   1-\tfrac{e^{-i(\theta -\theta ^{\prime })}}{iz}\big)}{(iz)^{-\bar
     A}}\bigg)\\
\notag&\qquad+\text{bounded terms}.
\end{align*}
In conclusion, since $\big\vert e^{\frac{r}{2}( z-\frac{ 1}{z})
}\big\vert =1$ on $\Gamma _{1}$ and $\log \big( 1-\frac{ e^{\pm i(\theta
-\theta ^{\prime })}}{iz}\big) \sim -\frac{e^{\pm i(\theta -\theta ^{\prime })}
}{iz}$ as $\left\vert z\right\vert \rightarrow \infty $, we obtain the
desired estimate
\begin{equation}  \label{eq:ultima}
\sup_{\substack{r,r'\geq0\\\theta,\theta'\in{\mathbb S}^1}
  }\left|K_1(r,r',\theta,\theta')\right|<\infty,
\end{equation}
 which together with \eqref{eq:decomp}
and \eqref{eq:k2} proves claim \eqref{eq:claim}.
The proof now follows by \eqref{eq:finita}, \eqref{eq:decab},
\eqref{eq:abest} and \eqref{eq:claim}.

In the resonant case $\widetilde A\in\frac12\Z$, we can repeat exactly the same
arguments as above, using Lemmas \ref{l:21caseelect-nonsym-A} and
\ref{l:21half} instead of Lemma \ref{lem:spectral}; although the
control on the remainder terms of the asymptotic expansion is in this
case less strong than in the non-resonant case, it is easy to verify
that it is enough both for \eqref{eq:7} and to estimate
$\sup|K_2|$ with $C\sum_{|j|\geq\ell}|j|^{-\frac{4}3}<\infty$ in order
to ensure
\eqref{eq:k2}.

\appendix
\section{A perturbation lemma}
The following result is based on Kato's Perturbation Theory (see \cite{K}).

\begin{Lemma}\label{lem:kato}
  Let $L_0,L:\mathcal H\to\mathcal H$ be two self-adjoint operators on a Hilbert space $\mathcal H$.
  Denote by
  \begin{equation*}
    R_0(\lambda):=(L_0-\lambda I)^{-1}
    \qquad
    R(\lambda):=(L-\lambda I)^{-1}.
  \end{equation*}
  Then:
  \begin{enumerate}
    \item
    if $R_0(\lambda), R(\lambda)\in\mathcal L(\mathcal H)$, then $\lambda$ is not an eigenvalue (neither for $R_0$ nor for $R(\lambda)$);
    \item
    if the operator
    \begin{equation*}
      T:=\frac1{2\pi i}\int_{\Gamma}(R(\lambda)-R_0(\lambda))\,d\lambda
    \end{equation*}
    has operator norm $\|T\|_{\mathcal L(\mathcal H)}<1$, being $\Gamma$ a closed curve in the complex plane, then the number of eigenvalues (counted with multiplicity) of $L_0$ and $L$ contained in the region bounded by $\Gamma$ is the same.
  \end{enumerate}
\end{Lemma}

\section{Asymptotics of eigenvalues and eigenfunctions in the purely
  electric case}\label{sec:asympt-eigenv-eigenf}

For the sake of completeness we sketch in this appendix a proof of asymptotics of eigenvalues and
eigenfunctions in the purely electric case. We consider the problem
\begin{equation}\label{a1}
\begin{cases}
-y^{\prime \prime }+a(x)y =\lambda y,\quad\text{in }[0,2\pi],  \\
y(0) =y(2\pi ),\\
y^{\prime }(0) =y^{\prime }(2\pi ).
\end{cases}
\end{equation}
\subsection{The case of a symmetric potential}
Let us  first assume that $a(x)\in L^2({\mathbb S}^1)$ is symmetric with respect to $x=\pi$,
i.e.
\begin{equation}\label{eq:asym}
a(\pi
-s)=a(\pi +s),\quad\text{for all }s\in[0,\pi].
\end{equation}
For every $k\in\N$, $k\geq1$, let us consider the space
\[
E_k=\{v\in
L^2({\mathbb S}^1): \ (v,\sin(kt))_{L^2}=0\text{ and } (v,\cos(jt))_{L^2}=0\text{ for all }j\in\N\}.
\]
We denote as $\widetilde
a=\frac1{2\pi}\int_0^{2\pi}a(t)\,dt$ the average of $a$.
For every $j\in \N$ we denote as
\[
a_{c,j}=\frac1\pi\int_0^{2\pi}a(t)\cos(jt)\,dt
\]
the $j$-th cos-Fourier coefficient of $a(x)$.

 We observe that
since
$
\int_0^{2\pi}\big(\widetilde a-a(x)-\tfrac12
a_{c,2k}\big)\sin(kx)\sin(kx)\,dx=0
$
by the definition of $a_{c,2k}$, and $\int_0^{2\pi}\big(\widetilde a-a(x)-\tfrac12 a_{c,2k}\big)\sin(kx)\cos(kx)\,dx=0$
by the symmetry condition \eqref{eq:asym}, then the problem
\begin{equation}\label{y0}
\begin{cases}
-\varphi''(x)-k^2\varphi(x)=\big(\widetilde a-a(x)-\tfrac12
a_{c,2k}\big)\sin(kx),\quad\text{in }[0,2\pi],\\
\varphi(0)=\varphi(2\pi),\\
\varphi'(0)=\varphi'(2\pi),
\end{cases}
\end{equation}
admits solutions.

\begin{Lemma}\label{l:B1}
Let $a\in H^{\frac{1}{2}+\delta }([0,2\pi])$ for some $0<\delta \ll
1$ such that \eqref{eq:asym} holds. Then, for every $k\in\N$, $k\geq1$, there exists a unique
$\varphi_k\in E_k$ solving (\ref{y0}). Furthermore there exists some $C>0$
independent of $k$ such that
\begin{equation}
\left\Vert \varphi_k\right\Vert _{L^{2}} \leq C\, \frac{\left\Vert
a\right\Vert _{L^{2}}}{k}, \quad
\left\Vert \varphi_k\right\Vert _{L^{\infty }} \leq C\,\frac{\left\Vert
a\right\Vert _{H^{\frac{1}{2}+\delta }}}{k}.  \label{d3}
\end{equation}
\end{Lemma}

\begin{proof}
The unique solution to \eqref{y0} in $E_k$, i.e. of the form
\[
\varphi_k(x)=\sum_{\substack{j=1\\j\neq k}}^\infty b_{j}\sin (jx),
\]
is given by
\[
\varphi_k(x)=\frac1\pi\sum_{\substack{j=1\\j\neq k}}^\infty
\frac{\int_0^{2\pi}(\widetilde a-a(t) )\sin (kt)\sin (jt)\,dt}{j^{2}-k^{2}}
\sin (jx).
\]
We notice that
\begin{align*}
\varphi_k(x)&=\frac1\pi\sum_{\substack{j=1\\j\neq k}}^\infty
\frac{\int_0^{2\pi}(\widetilde a-a(t) )\sin (kt)\sin (jt)\,dt}{j^{2}-k^{2}}
\sin (jx)\\
&=\frac1{2\pi}\sum_{\substack{j=1\\j\neq k}}^\infty
\frac{\int_0^{2\pi}a(t)
\big(\cos((k-j)t)-\cos((k+j)t)
\big)
\,dt}{k^{2}-j^{2}}
\sin (jx)\\
&=\frac{1}{2 }\sum_{\substack{j=1\\j\neq k}}^\infty
\frac{a_{c,\left\vert k-j\right\vert }
-a_{c,k+j}}{(k-j)(k+j)}
\sin (jx).
\end{align*}
Therefore
\[
\left\Vert \varphi_k\right\Vert _{L^{2}}^{2}=\frac\pi4 \sum_{\substack{j=1\\j\neq k}}^\infty \frac{\left(
a_{c,\left\vert k-j\right\vert }-a_{c,k+j}\right) ^{2}}{(k-j)^{2}(k+j)^{2}}
\leq \frac{C}{k^{2}}\left\Vert a\right\Vert _{L^{2}}^{2}
\]
for some positive constant $C>0$ independent of $k$.
Finally we can estimate
\begin{align*}
\left\vert\varphi_k(x)\right\vert  &\leq
\frac{1}{2 }\sum_{\substack{j=1\\j\neq k}}^\infty
\frac{|a_{c,\left\vert k-j\right\vert }|+|a_{c,k+j}|}{|j^2-k^2|}\\
&\leq \bigg( \sum_{j\neq k}\left\vert k-j\right\vert ^{1+2\delta
}\left\vert a_{c,\left\vert k-j\right\vert }\right\vert ^{2}\bigg) ^{\!\!\frac{1
}{2}}\bigg( \sum_{j\neq k}\frac{1}{\left\vert j^{2}-k^{2}\right\vert
^{2}\left\vert k-j\right\vert ^{1+2\delta }}\bigg) ^{\!\!\frac{1}{2}}\\
&\quad+\bigg(
\sum_{j\neq k}\left\vert k+j\right\vert ^{1+2\delta }\left\vert
a_{c,\left\vert k+j\right\vert }\right\vert ^{2}\bigg) ^{\!\!\frac{1}{2}}\bigg(
\sum_{j\neq k}\frac{1}{\left\vert j^{2}-k^{2}\right\vert ^{2}\left\vert
k+j\right\vert ^{1+2\delta }}\bigg) ^{\!\!\frac{1}{2}} \\
&\leq \frac{C\| a\|_{H^{\frac12+\delta }}}{k}
\end{align*}
for some positive constant $C>0$ independent of $k$. The proof is then complete.
\end{proof}

\begin{Lemma}\label{l:B2}
  For every $k\geq1$  and $f\in E_k$
there exists a unique $\tilde\varphi_{f,k}\in E_k$ solving
\begin{equation}\label{y0t}
\begin{cases}
-\tilde\varphi_{f,k}''(x)-k^2\tilde\varphi_{f,k}(x)=f(x),\quad\text{in }[0,2\pi],\\
\tilde\varphi_{f,k}(0)=\tilde\varphi_{f,k}(2\pi),\\
\tilde\varphi_{f,k}'(0)=\tilde\varphi_{f,k}'(2\pi).
\end{cases}
\end{equation}
 Furthermore
\begin{equation}\label{eq:6l2}
  \left\Vert \tilde\varphi_{f,k}\right\Vert _{L^{2}} \leq \frac{\left\Vert
      f\right\Vert _{L^{2}}}{k},
\end{equation}
$\tilde\varphi_{f,k}\in L^\infty({\mathbb S}^1)$ and for every
$\delta\in(0,1)$ there exists $C_\delta>0$  independent of $k$ and $f$ such that
\begin{equation} \label{e3}
  \left\Vert \tilde\varphi_{f,k}\right\Vert _{L^{\infty }} \leq
  \frac{C_\delta}{k^{\frac{1-\delta}2 }}\left\Vert f\right\Vert _{L^2}.
\end{equation}
\end{Lemma}

\begin{proof}
Since $f\in E_k$,  equation \eqref{y0t} is solvable;
furthermore, looking for solutions in Fourier series, it turns out
that there exists a unique solution in $E_k$
which is given by
\[
\tilde\varphi_{f,k}(x)=
\sum_{\substack{j=1\\j\neq k}}^\infty \frac{
f_{s,j}\sin(jx)}{j^{2}-k^{2}}
\]
where
\[
f_{s,j}=\frac1\pi\int_0^{2\pi}f(t)\sin(jt)\,dt.
\]
Hence
\[
\left\Vert \tilde\varphi_{f,k}\right\Vert _{L^{2}}^{2}=
\pi
\sum_{\substack{j=1\\j\neq k}}^\infty \frac{
f_{s,j}^2}{(j^{2}-k^{2})^2} \leq
\frac{\pi}{k^{2}}\sum_{\substack{j=1\\j\neq k}}^\infty
f_{s,j}^2=\frac{\left\Vert
f\right\Vert _{L^{2}}^{2}}{k^{2}}
\]
and \eqref{eq:6l2} is proved.
Let $\delta\in(0,1)$. We have that
\begin{equation*}
|\tilde\varphi_{f,k}(x)|\leq
\sum_{\substack{j=1\\j\neq k}}^\infty \frac{
|f_{s,j}|}{|j^{2}-k^{2}|}\leq
\bigg(\sum_{\substack{j=1\\j\neq k}}^\infty |f_{s,j}|^2\bigg)^{\!\!1/2}
\bigg(\sum_{\substack{j=1\\j\neq k}}^\infty
\frac{1}{|j^{2}-k^{2}|^2}\bigg)^{\!\!1/2}\leq C_\delta\frac{\|f\|_{L^2}}{k^{\frac{1-\delta}{2}}}
\end{equation*}
thus completing the proof.
\end{proof}

\begin{Proposition}\label{p:asy1elec}
Let $a\in H^{\frac{1}{2}+\delta }([0,2\pi])$ for some $0<\delta \ll
1$ such that \eqref{eq:asym} holds. Then, for every $k$ sufficiently
large, there exists
an eigenfunction $y_k^s$ of problem \eqref{a1} corresponding to an
eigenvalue $\lambda_k^s$ such that
\begin{align*}
  &y_k^s(x)=\sin (kx)+R_k(x), \quad\text{with }\|R_k\|_{L^\infty}=
  O\Big(\frac1k\Big) \text{ as }
  k\to+\infty. \\
  &\lambda_k^s =k^2+\widetilde a -\frac{1}{2 }a_{c,2k}+O\Big(\frac1k\Big)\text{ as }
  k\to+\infty.
\end{align*}
\end{Proposition}

\begin{proof}

If $k\in\N$, $k\geq1$, and $\varphi\in E_k$, we define
\begin{equation}\label{l}
  \widetilde\lambda(\varphi,k)=-\frac{1}{2 }a_{c,2k}+
  \frac{1}{\pi }\int_0^{2\pi} a(x)(\varphi(x)+\varphi_k(x))\sin (kx)dx
\end{equation}
and
\begin{equation}\label{f}
F(\varphi,k)=\Big(\tfrac{1}{2
}a_{c,2k}+\widetilde\lambda(\varphi,k)\Big) \sin(kx)+\big(\widetilde\lambda(\varphi,k)+\overline{a}-a(x)\big)
\big(\varphi(x)+\varphi_k(x)\big).
\end{equation}
Since $\varphi,\varphi_k\in E_k$, from \eqref{l} it follows that
$(F(\varphi,k),\sin(kt))_{L^2}=0$, while from the symmetry assumption
\eqref{eq:asym} we deduce that
  $(F(\varphi,k),\cos(jt))_{L^2}=0$ for all $j$, so that
  $F(\varphi,k)\in E_k$. Then Lemma \ref{l:B2} ensures the existence of
a unique $T_k(\varphi)\in E_k$ satisfying
\[
\begin{cases}
-(T_k(\varphi))''-k^2T_k(\varphi)=F(\varphi,k) ,\quad\text{in }[0,2\pi],\\
T_k(\varphi)(0)=T_k(\varphi)(2\pi),\quad
T_k(\varphi)'(0)=T_k(\varphi)'(2\pi).
\end{cases}
\]
Furthermore from \eqref{eq:6l2} it follows that
\begin{equation*}
  \left\Vert T_k(\varphi)\right\Vert _{L^{2}} \leq \frac{\left\Vert
      F(\varphi,k) \right\Vert _{L^{2}}}{k}.
\end{equation*}
We notice that
\[
|\widetilde\lambda(\varphi,k)|\leq
\frac{\|a\|_{L^2}}{\pi}\big(1+\|\varphi\|_{L^2}+\|\varphi_k\|_{L^2}\big)
\]
and, for some $C>0$ independent of $k$ and $\varphi$,
\begin{equation}\label{eq:stF}
\|      F(\varphi,k) \|_{L^{2}}\leq C\|a\|_{L^\infty}(1+\|\varphi\|_{L^2}+\|\varphi_k\|_{L^2})
(\|\varphi\|_{L^2}+\|\varphi_k\|_{L^2}).
\end{equation}
The above estimate together with Lemma \ref{l:B1}  imply that,
letting
\[
B_k=\Big\{\varphi\in E_k:\
\|\varphi\|_{L^2}\leq\frac{1}{k^{1-\delta}}\Big\},
\]
 $T(\varphi)\in B_k$
for every $\varphi\in B_k$. Furthermore it is easy to verify that $T$
is a contraction in $B_k$ endowed with the $L^2$-metric. Then
$T$ admits a unique fixed point in $B_k$, i.e. there exists a unique
$\widetilde \varphi_k\in B_k$ such that $T(\widetilde
\varphi_k)=\widetilde \varphi_k$. Then
\begin{equation}\label{eq:eqt}
\begin{cases}
  -\widetilde \varphi_k''-k^2\widetilde \varphi_k=\Big(\tfrac{a_{c,2k}}{2
  }+\widetilde\lambda(\widetilde \varphi_k,k)\Big)
  \sin(kx)+\big(\widetilde\lambda(\widetilde \varphi_k,k)+\overline{a}-a\big)
  \big(\widetilde \varphi_k+\varphi_k\big),\hskip-15pt\\
  \widetilde \varphi_k(0)=\widetilde \varphi_k(2\pi),\quad \widetilde
  \varphi_k'(0)=\widetilde \varphi_k'(2\pi).
\end{cases}
\end{equation}
From estimate \eqref{e3} of Lemma \ref{l:B2}, \eqref{eq:stF}, and
Lemma \ref{l:B1}, it follows that, if $\delta\leq1/3$ and $k$ is
sufficiently large
\begin{align}\label{eq:inftilphi}
\|\widetilde \varphi_k\|_{L^\infty}&=\|T(\widetilde
\varphi_k)\|_{L^\infty}\leq C_\delta\frac{\|F(\widetilde
  \varphi_k,k)\|_{L^2}}{k^{\frac{1-\delta}{2}}}\\
\notag&\leq
\frac{C_\delta}{k^{\frac{1-\delta}{2}}} C\|a\|_{L^\infty}(1+\|\widetilde \varphi_k\|_{L^2}+\|\varphi_k\|_{L^2})
(\|\widetilde \varphi_k\|_{L^2}+\|\varphi_k\|_{L^2})\\
\notag&\leq \frac{\text{\rm const}}{k^{\frac{1-\delta}{2}+(1-\delta)}}\leq
\frac{\text{\rm const}}{k}
\end{align}
for some positive ${\rm const}>0$ independent of $k$. In particular
\eqref{eq:inftilphi} implies that
\begin{equation}\label{eq:2lphi}
  \|\widetilde \varphi_k\|_{L^2} \leq
\frac{\text{\rm const}}{k}
\end{equation}
 for some positive ${\rm const}>0$ independent of $k$.
From \eqref{l}, \eqref{eq:2lphi}, and \eqref{d3}
\begin{equation}\label{eq:asitildelam}
  \widetilde\lambda(\widetilde \varphi_k,k)+\frac{1}{2 }a_{c,2k}=
  \frac{1}{\pi }\int_0^{2\pi} a(x)(\widetilde
  \varphi_k(x)+\varphi_k(x))\sin (kx)dx=O\Big(\frac1k\Big)
\end{equation}
as $k\to+\infty$.
Let us consider the function
\[
y_k(x)=\sin(kx)+\varphi_k(x)+\widetilde\varphi_k(x),\quad
x\in[0,2\pi].
\]
From the fact that $\varphi_k$ solves \eqref{d3} and \eqref{eq:eqt} we
obtain that $y_k$ satisfies
\begin{equation}\label{ykeq}
\begin{cases}
-y_k''(x)+a(x)y_k(x)=(k^2+\widetilde a+\widetilde\lambda(\widetilde \varphi_k,k))y_k(x),\quad\text{in }[0,2\pi],\\
y_k(0)=y_k(2\pi),\\
y_k'(0)=y_k'(2\pi),
\end{cases}
\end{equation}
i.e. $y_k$ is an eigenfunction of problem \eqref{a1} corresponding to
the eigenvalue
\[
\lambda_k=k^2+\widetilde a+\widetilde\lambda(\widetilde \varphi_k,k).
\]
From \eqref{d3} and \eqref{eq:inftilphi} we have that
$\|y_k(x)-\sin(kx)\|_{L^\infty}=O(\frac1k)$ as $k\to+\infty$, while
\eqref{eq:asitildelam} implies that
$\lambda_k=k^2+\widetilde a -\frac{1}{2 }a_{c,2k}+O(\frac1k)$ as
$k\to+\infty$, thus completing the proof.
\end{proof}

Arguing in an analogous way, one can prove the following statement.

\begin{Proposition}\label{p:asy2elec}
Let $a\in H^{\frac{1}{2}+\delta }([0,2\pi])$ for some $0<\delta \ll
1$ such that \eqref{eq:asym} holds. Then, for every $k$ sufficiently
large, there exists
an eigenfunction $y_k^c$ of problem \eqref{a1} corresponding to an
eigenvalue $\lambda_k^c$ such that
\begin{align*}
  &y_k^c(x)=\cos (kx)+R_k(x), \quad\text{with }\|R_k\|_{L^\infty}=
  O\Big(\frac1k\Big) \text{ as }
  k\to+\infty. \\
  &\lambda_k^c =k^2+\widetilde a +\frac{1}{2 }a_{c,2k}+O\Big(\frac1k\Big)\text{ as }
  k\to+\infty.
\end{align*}
\end{Proposition}

\begin{remark}\label{r:w1in}
If $a\in W^{1,\infty }([0,2\pi])$, then by integrating by parts it
follows that
\[
\left\vert a_{c,2k}\right\vert =\frac{1}{2k\pi}\left\vert \int_0^{2\pi} a^{\prime
}(x)\sin (2kx)dx\right\vert \leq \frac{1}{k}\left\Vert
a\right\Vert _{W^{1,\infty }}.
\]
\end{remark}

\begin{remark}\label{r:rest}
We notice that, with similar estimates as the ones performed in the
proofs of Lemmas  \ref{l:B1} and \ref{l:B2}, the derivative of the
remainder terms $R_k$ of Propositions \ref{p:asy1elec} and
\ref{p:asy2elec} can be estimated as $\|R_k'\|_{L^\infty}=
  O\big(k^\frac{1+\delta}2\big)$.
\end{remark}

Combining Propositions \ref{p:asy1elec} and \ref{p:asy2elec} and
Remark \ref{r:w1in}, and
arguing as in lemmas \ref{l:usekato} and \ref{l:lambdainball} we obtain
the following result.

\begin{Lemma}\label{l:21caseelect}
  Let  $a\in W^{1,\infty}({\mathbb S}^1)$ satisfying \eqref{eq:asym}
  and  $\widetilde
a=\frac1{2\pi}\int_0^{2\pi}a(s)\,ds$. Let $\{\mu_k\}_{k\geq1}$ be
the eigenvalues of the operator $\frak L_{0,a}\varphi=
-\varphi''+a\varphi$ with $2\pi$-periodic boundary conditions.

Then there exist  $k^*,\ell\in\N$ such that $\{\mu_k: k>
k^*\}=\{\lambda_j: j\in\Z,|j|\geq \ell\}$ and
\[
\lambda_j=\widetilde a+j^2+O\big(\tfrac1{|j|}\big),\quad\text{as }|j|\to+\infty.
\]
Furthermore, for all $j\in\Z$, $|j|\geq\ell$, there exists
 a $L^{2}\big({\mathbb{S}}^{1},{\mathbb{C}}\big)$-normalized
eigenfunction  $\phi_j$
of the operator $\frak L_{0,a}$
corresponding to
 the  eigenvalue $\lambda_j$ such that
\[
  \phi_j(\theta)=\begin{cases}
\frac1{\sqrt{\pi}}\Big(\sin(j\theta)+R_j(\theta)\Big),&\text{if }j>0,\\
\frac1{\sqrt{\pi}}\Big(\cos(j\theta)+R_j(\theta)\Big),&\text{if }j<0,
\end{cases}
\]
where $\|R_j\|_{L^\infty({\mathbb S}^1)}= O\big(\tfrac1{|j|}\big)$ as
$|j|\to\infty$.
\end{Lemma}

\subsection{The general case of a non-symmetric potential}
Let us now drop assumption \eqref{eq:asym}.
Arguing an in proof of Lemma \ref{l:usekato}, one can prove that there
exist some $c>0$ and $k_0\in\N$ such that large eigenvalues of $\frak L_{0,a}$  are
contained in $\bigcup_{j=j_0}^\infty
B\big(j^2,c\big)$ and, if $j\geq j_0$, each ball $B\big(j^2,c)$ contains exactly two
eigenvalues (counted with their own multiplicity).

Let   $\varphi$ be a  $L^{2}\big({\mathbb{S}}^{1},{\mathbb{C}}\big)$-normalized
eigenfunction of $\frak L_{0,a}$
associated to the eigenvalue $\mu\in B\big(j^2,c)$, i.e.
\begin{equation*}
\begin{cases}
-\varphi''(\theta)+a(\theta ) \varphi(\theta)=\mu
\varphi(\theta),\quad\text{in }[0,2\pi],\\
\varphi(0) =\varphi (2\pi ), \\
\varphi^\prime(0) =\varphi^{\prime }(2\pi ).
\end{cases}
\end{equation*}
From the mean value theorem, there exists $\bar\theta_\varphi\in (0,2\pi)$
such that $\varphi'(\bar\theta_\varphi)=0$. The function
\[
\widetilde\varphi(\theta)=\varphi(\theta+\bar\theta_\varphi)
\]
satisfies
\begin{equation*}
\begin{cases}
-\widetilde\varphi''(\theta)+a(\theta+\bar\theta_\varphi ) \widetilde\varphi(\theta)=\mu
\widetilde\varphi(\theta),\quad\text{in }[0,2\pi],\\
\widetilde\varphi(0) =\widetilde\varphi (2\pi ), \\
\widetilde\varphi^\prime(0) =\widetilde\varphi^{\prime }(2\pi )=0.
\end{cases}
\end{equation*}
Hence the function
\[
\widehat\varphi(\theta)=
\begin{cases}
  \widetilde\varphi(2\theta),&\text{ if }\theta\in[0,\pi],\\
  \widetilde\varphi(4\pi-2\theta),&\text{ if }\theta\in[\pi,2\pi],\\
\end{cases}
\]
is smooth and satisfies
\begin{equation*}
\begin{cases}
-\widehat\varphi''(\theta)+\widehat a(\theta) \widehat\varphi(\theta)=4\mu
\widehat\varphi(\theta),\quad\text{in }[0,2\pi],\\
\widehat\varphi(0) =\widehat\varphi (2\pi ), \\
\widehat\varphi^\prime(0) =\widehat\varphi^{\prime }(2\pi )=0,
\end{cases}
\end{equation*}
where
\[
\widehat a(\theta)=
\begin{cases}
  4a(2\theta+\bar\theta_\varphi),&\text{ if }\theta\in[0,\pi],\\
  4 a(4\pi-2\theta+\bar\theta_\varphi),&\text{ if }\theta\in[\pi,2\pi].\\
\end{cases}
\]
We notice that $\widehat a\in W^{1,\infty}({\mathbb S}^1)$,
$\widehat a$ satisfies \eqref{eq:asym}, and
$\frac1{2\pi}\int_0^{2\pi}\widehat
a=\frac4{2\pi}\int_0^{2\pi}a=4\widetilde a$. From Propositions
\ref{p:asy1elec} and \ref{p:asy2elec} and Lemma \ref{l:21caseelect},
we have that, for some large $k$,
\begin{equation}\label{eq:5}
4\mu=k^2+4\widetilde a+O(1/|k|)
\end{equation}
and $\widehat\varphi$ is a linear combination of the functions $\cos
(k\theta)+R_k^1(\theta)$ and $\sin (k\theta)+R_k^2(\theta)$, with $\|R_k^i\|_{L^\infty}=
  O\big(\frac1k\big)$ as $k\to+\infty$. The condition
  $\widehat\varphi^\prime(0)=0$ together with Remark
  \ref{r:rest} implies that $\widehat\varphi$ is a multiple of $\cos
  (k\theta)+R_k^1(\theta)$, whereas the fact that $\widehat\varphi(0)=\widehat\varphi(\pi)$
implies that necessarily $k$ is even. From the evenness of $k$,
\eqref{eq:5}, and the fact that $\mu\in B\big(j^2,c)$, we conclude
that (provided $j$ is large enough) $k=\pm2j$, so that
\[
\mu=j^2+\widetilde a+O(1/|j|)
\]
and $\widehat\varphi$ is a multiple of $\cos
  (2j\theta)+R_{\pm2j}^1(\theta)$. Then
\[
\varphi(\theta)=\frac1{\sqrt\pi}\cos(j(\theta-\bar\theta_\varphi))+R_j(\theta)
\]
with $\|R_j\|_{L^\infty}=
  O\big(\frac1{|j|}\big)$ as $j\to+\infty$. The above argument proves
  the following extension to the non-symmetric case of Lemma
  \ref{l:21caseelect}.

\begin{Lemma}\label{l:21caseelect-nonsym}
  Let  $a\in W^{1,\infty}({\mathbb S}^1)$ and $\widetilde
a=\frac1{2\pi}\int_0^{2\pi}a(s)\,ds$. Let $\{\mu_k\}_{k\geq1}$ be
the eigenvalues of the operator $\frak L_{0,a}\varphi=
-\varphi''+a\varphi$ with $2\pi$-periodic boundary conditions.
Then there exist  $k^*,\ell\in\N$ such that $\{\mu_k: k>
k^*\}=\{\lambda_j: j\in\Z,|j|\geq \ell\}$ and
\[
\lambda_j=\widetilde a+j^2+O\big(\tfrac1{|j|}\big),\quad\text{as }|j|\to+\infty.
\]
Furthermore, for all $j\in\Z$, $|j|\geq\ell$, there exist some $\theta_j\in[0,2\pi]$ and
 a $L^{2}\big({\mathbb{S}}^{1},{\mathbb{C}}\big)$-normalized
eigenfunction  $\phi_j$ of the operator $\frak L_{0,a}$ on $\mathbb{S}^{1}$
corresponding to
 the  eigenvalue $\lambda_j$  such that
\[
  \phi_j(\theta)=
\frac1{\sqrt{\pi}}\cos(j(\theta-\theta_j))+R_j(\theta)
\]
where $\|R_j\|_{L^\infty({\mathbb S}^1)}= O\big(\tfrac1{|j|}\big)$ as
$|j|\to\infty$.
\end{Lemma}

Combining Lemma \ref{l:21caseelect-nonsym} with Lemma \ref{l:probequiv} we
obtain the following result which is the analogue of Lemma
\ref{lem:spectral} in the case $\widetilde A\in\Z$ (we notice that if
$\widetilde A\in\Z$ then $\bar A=0$).

\begin{Lemma}\label{l:21caseelect-nonsym-A}
  Let  $a\in W^{1,\infty}({\mathbb S}^1)$, $\widetilde
a=\frac1{2\pi}\int_0^{2\pi}a(s)\,ds$, $A\in W^{1,\infty}({\mathbb
  S}^1)$ such that  $\widetilde
A=\frac1{2\pi}\int_0^{2\pi}A(s)\,ds\in\Z$. Let $\{\mu_k\}_{k\geq1}$ be
the eigenvalues of the operator $L_{{\mathbf{A}},a}$ defined in
\eqref{eq:laplacebeltrami} with ${\mathbf A}$ as in \eqref{eq:AA}.

Then there exist  $k^*,\ell\in\N$ such that $\{\mu_k: k>
k^*\}=\{\lambda_j: j\in\Z,|j|\geq \ell\}$ and
\[
\lambda_j=\widetilde a+j^2+O\big(\tfrac1{|j|}\big),\quad\text{as }|j|\to+\infty.
\]
Furthermore, for all $j\in\Z$, $|j|\geq\ell$, there exist some $\theta_j\in[0,2\pi]$ and
 a $L^{2}\big({\mathbb{S}}^{1},{\mathbb{C}}\big)$-normalized
eigenfunction  $\phi_j$ of the operator $L_{{\mathbf{A}},a}$ on $\mathbb{S}^{1}$
corresponding to
 the  eigenvalue $\lambda_j$ such that
\[
  \phi_j(\theta)=
\frac1{\sqrt{\pi}}e^{-i\int_0^\theta
  A(t)\,dt}\cos(j(\theta-\theta_j))+R_j(\theta),
\]
where $\|R_j\|_{L^\infty({\mathbb S}^1)}= O\big(\tfrac1{|j|}\big)$ as
$|j|\to\infty$.
\end{Lemma}

\subsection{The case of half-integer circulation}

From the asymptotics of eigenvalues and eigenfunctions in the case of
integer circulation discussed above, we can derive the asymptotics of
eigenvalues and eigenfunctions also in the case of magnetic potentials
with half-integer circulation, corresponding to the case
$\widetilde A\in \frac12\Z\setminus\Z$.

\begin{Lemma}\label{l:21half}
  Let  $a\in W^{1,\infty}({\mathbb S}^1)$, $\widetilde
a=\frac1{2\pi}\int_0^{2\pi}a(s)\,ds$, $A\in W^{1,\infty}({\mathbb
  S}^1)$ such that  $\widetilde
A=\frac1{2\pi}\int_0^{2\pi}A(s)\,ds\in\frac12\Z\setminus \Z$. Let $\{\mu_k\}_{k\geq1}$ be
the eigenvalues of the operator $L_{{\mathbf{A}},a}$ defined in
\eqref{eq:laplacebeltrami} with ${\mathbf A}$ as in \eqref{eq:AA}.
Then there exist  $k^*,\ell\in\N$ such that $\{\mu_k: k>
k^*\}=\{\lambda_j: j\in\Z,|j|\geq \ell\}$ and
\[
\lambda_j=\widetilde a+\bigg(j+\frac12\bigg)^{\!\!2}+O\big(\tfrac1{|j|}\big),\quad\text{as }|j|\to+\infty.
\]
Furthermore, for all $j\in\Z$, $|j|\geq\ell$, there exist some
$\theta_j\in[0,2\pi]$ and
 a $L^{2}\big({\mathbb{S}}^{1},{\mathbb{C}}\big)$-normalized
eigenfunction  $\phi_j$ of the operator $L_{{\mathbf{A}},a}$ on $\mathbb{S}^{1}$
corresponding to
 the  eigenvalue $\lambda_j$ such that
\[
  \phi_j(\theta)=\frac1{\sqrt{\pi}}e^{-i\int_0^\theta
  A(t)\,dt}\Big(\cos\big((j+\tfrac12)(\theta-\theta_j)\big)+R_j(\theta)\Big),
\]
where $\|R_j\|_{L^\infty({\mathbb S}^1)}= O\big(\tfrac1{|j|}\big)$ as
$|j|\to\infty$.
\end{Lemma}
\begin{proof}
 If $\widetilde
A\in\frac12\Z\setminus \Z$, we have that
 if $\varphi$ is an eigenfunction of $\frak L_{A,a}$
associated to the eigenvalue $\mu$ then
$\widetilde \varphi(\theta)=e^{i\int_0^\theta
  A(s)\,ds}\varphi(\theta)$ satisfies
\begin{equation*}
\begin{cases}
-\frac{d^{2}\widetilde \varphi}{d\theta ^{2}}+a(\theta )\widetilde \varphi=\mu \widetilde \varphi
\\
\widetilde \varphi(0) =-\widetilde \varphi (2\pi ) \\
\widetilde \varphi^\prime(0) =-\widetilde \varphi^{\prime }(2\pi ),
\end{cases}
\end{equation*}
and therefore
\[
\widehat \varphi(\theta)=
\begin{cases}
  \widetilde \varphi(2\theta),&\text{if }\theta\in[0,\pi],\\
  -\widetilde \varphi(2\theta-2\pi),&\text{if }\theta\in[\pi,2\pi],\\
\end{cases}
\]
satisfies
\begin{equation}\label{eq:6}
\widehat \varphi(s+\pi)=-\widehat \varphi(s), \quad\widehat
\varphi'(s+\pi)=-\widehat \varphi'(s)\quad\text{for all
  $s\in[0,2\pi]$}
\end{equation}
and
\begin{equation*}
\begin{cases}
-\frac{d^{2}\widehat \varphi}{d\theta ^{2}}+\widehat a(\theta )\widehat \varphi=4\mu \widehat \varphi
\\
\widehat \varphi(0) =\widehat \varphi (2\pi ) \\
\widetilde \varphi^\prime(0) =\widehat \varphi^{\prime }(2\pi ),
\end{cases}
\end{equation*}
where
\[
\widehat a(\theta)=
\begin{cases}
  4a(2\theta),&\text{if }\theta\in[0,\pi],\\
  4a(2\theta-2\pi),&\text{if }\theta\in[\pi,2\pi].
\end{cases}
\]
We notice that if $a\in W^{1,\infty}({\mathbb S}^1)$ then also
$\widehat a\in W^{1,\infty}({\mathbb S}^1)$.

From Lemma \ref{l:21caseelect-nonsym}, we have that, if $\mu$
if sufficiently large, then $4\mu=k^2+4\widetilde a+O(\frac1k)$ for some
$k$ large and the eigenspace of the operator $-\frac{d^{2}}{d\theta
  ^{2}}+\widehat a$
 associated to $4\mu$ is generated by
one or two functions of the form
$\cos(k(\theta-\theta_k))+R_k(\theta)$ for some $\theta_k\in[0,2\pi]$ with
$\|R_k\|_{L^\infty}=
  O\big(\frac1k\big)$ and $\|R_k'\|_{L^\infty}=
  O\big(k^\frac{1+\delta}2\big)$. Since
condition \eqref{eq:6} can be satisfied only for odd
$k$, we conclude large eigenvalues of $\frak L_{A,a}$  are of the form
\[
\bigg(k+\frac12\bigg)^{\!\!2}+\widetilde a+O\bigg(\frac1k\bigg),\quad\text{as }k\to\infty
\]
with eigenfunctions of the form $=e^{-i\int_0^\theta
  A(s)\,ds}\cos((k+1/2)(\theta-\theta_k))+R_k(\theta)$
with $\|R_k\|_{L^\infty}=
  O\big(\frac1k\big)$.
The conclusion then follows
arguing as in lemmas \ref{l:usekato} and \ref{l:lambdainball}.
\end{proof}

\end{document}